%% file: oscd.tex
\documentclass[a4paper,11pt]{article}
\usepackage[utf8]{inputenc}
\usepackage{amsthm}
\usepackage{amsmath}
\usepackage{amssymb}
\usepackage{graphicx}
\usepackage{enumitem}
\usepackage{authblk}
\usepackage{breakurl}
\usepackage[pdftitle={On orthogonal symmetric chain decompositions}]{hyperref}
\usepackage{placeins}
\usepackage[margin=2.5cm]{geometry}
\usepackage[usenames,dvipsnames,table]{xcolor}
\usepackage{cancel}
\usepackage{bookmark}
\usepackage{array}

\hypersetup{pdfauthor={Karl D\344ubel, Sven J\344ger, Torsten M\374tze, Manfred Scheucher}}

\input{figures/tikz_styles}

\usepackage{pifont}
\newcommand{\yes}{{\color{OliveGreen}\ding{51}}}
\newcommand{\no}{{\color{Red}\ding{55}}}
\newcommand{\idk}{{\color{Orange}\textbf{???}}}

\newcommand{\ol}{\overline}
\newcommand{\neck}[1]{\langle #1\rangle}

\newcommand{\DGKS}{D^{\mathrm{GKS}}}
\newcommand{\RGKS}{R^{\mathrm{GKS}}}
\newcommand{\DJo}{D^{\mathrm{J}}}
\newcommand{\RJo}{R^{\mathrm{J}}}

\newtheorem{theorem}{Theorem}

\newtheorem{lemma}{Lemma}
\newtheorem{proposition}[lemma]{Proposition}
\newtheorem{observation}[lemma]{Observation}

\theoremstyle{remark}
\newtheorem{remark}{Remark}

\newcommand{\NN}{\mathbb{N}}

\title{On orthogonal symmetric chain decompositions\footnote{An extended abstract of this paper has been submitted to Eurocomb 2019.}}
\author{Karl D\"aubel}
\author{Sven J\"ager}
\author{Torsten M\"utze}
\author{Manfred Scheucher}
\affil{Institut f\"ur Mathematik, Technische Universit\"at Berlin, Germany \\
  \small \texttt{\{daeubel,jaeger,muetze,scheucher\}@math.tu-berlin.de}}
\date{\vspace{-5ex}}

\begin{document}

\maketitle

\begin{abstract}
The \emph{$n$-cube} is the poset obtained by ordering all subsets of $\{1,\ldots,n\}$ by inclusion, and it can be partitioned into $\binom{n}{\lfloor n/2\rfloor}$ chains, which is the minimum possible number.
Two such decompositions of the $n$-cube are called \emph{orthogonal} if any two chains of the decompositions share at most a single element.
Shearer and Kleitman conjectured in~1979 that the $n$-cube has $\lfloor n/2\rfloor+1$ pairwise orthogonal decompositions into the minimum number of chains, and they constructed two such decompositions.
Spink recently improved this by showing that the $n$-cube has three pairwise orthogonal chain decompositions for~$n\geq 24$.
In this paper, we construct four pairwise orthogonal chain decompositions of the $n$-cube for~$n\geq 60$.
We also construct five pairwise \emph{edge-disjoint} symmetric chain decompositions of the $n$-cube for~$n\geq 90$, where edge-disjointness is a slightly weaker notion than orthogonality, improving on a recent result by Gregor, J\"ager, M\"utze, Sawada, and Wille.
\end{abstract}

\section{Introduction}

The \emph{$n$-dimensional cube $Q_n$}, or $n$-cube for short, is the poset obtained by taking all subsets of $[n]:=\{1,\dotsc,n\}$, and ordering them by inclusion.
This poset is sometimes also called the \emph{subset lattice} or the \emph{Boolean lattice}, and it is a fundamental and widely studied object in combinatorics.
For illustration, Figure~\ref{fig:Q4} shows the 4-cube.
In this figure and throughout this paper, we draw posets by their Hasse diagrams.

\begin{figure}
\centering
\input{figures/Q4.tex}
\caption{Hasse diagram of the 4-cube~$Q_4$, with three pairwise orthogonal chain decompositions into 6~chains, highlighted by thick solid, dashed, and dotted lines.}
\label{fig:Q4}
\end{figure}
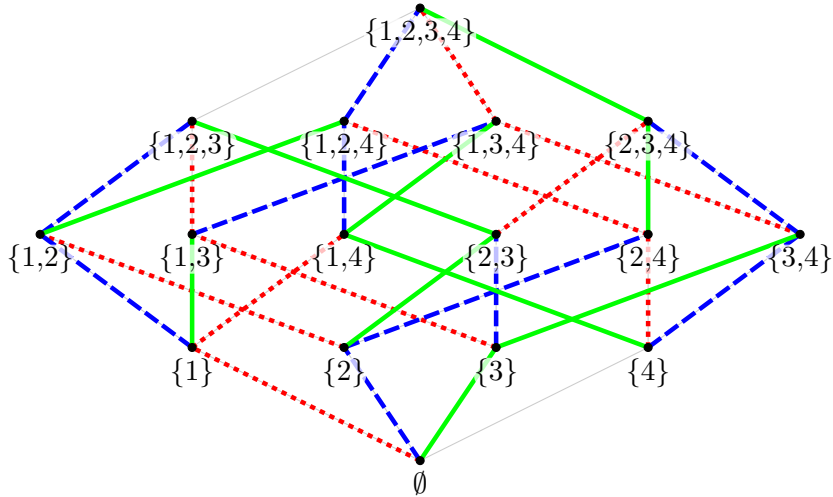

Clearly, $Q_n$ is a graded poset with rank function given by the set sizes, and every maximal chain has size~$n+1$.
We refer to the family of all subsets of a fixed size~$k\in\{0,\dotsc,n\}$ as the \emph{$k$th level} of~$Q_n$.
It is easy to see that $Q_n$ has a unique largest level $n/2$ for even~$n$, 
and two largest levels $\lfloor n/2\rfloor$ and $\lceil n/2\rceil$ for odd~$n$.
We refer to these levels as \emph{middle levels}.
Sperner's classical theorem~\cite{Sperner1928} asserts that each middle level is in fact a largest antichain of~$Q_n$, i.e., $Q_n$ has width $a_n:=\binom{n}{\lfloor n/2\rfloor}$.
As a consequence, at least $a_n$ many chains are needed to partition~$Q_n$, 
and by Dilworth's theorem~\cite{Dilworth1950}, a partition into this many chains indeed exists.
De Bruijn, van Ebbenhorst Tengbergen, and Kruiswijk~\cite{DeBruijnVETK1951} first described an inductive construction of a partition of~$Q_n$ into $a_n$ many chains that are all \emph{symmetric} and \emph{saturated}, i.e., every chain starts and ends in symmetric levels around the middle, and no chain skips any intermediate levels.
Throughout this paper, we will refer to their decomposition as the \emph{standard decomposition}.
Lewin~\cite{Lewin1972}, Aigner~\cite{Aigner1973}, and White and Williamson~\cite{WhiteWilliamson1977} gave alternative descriptions of the standard decomposition via greedy matching algorithms as well as explicit local rules to follow the chains in the standard decomposition.
The easiest-to-remember local rule using parenthesis matching was given by Greene and Kleitman~\cite{GreeneKleitman1976} (we will describe their rule in Section~\ref{sec:proof-unroll}).
The standard decomposition of~$Q_n$ was famously used by Kleitman~\cite{Kleitman1965} to prove the two-dimensional case of the Littlewood-Offord conjecture on signed sums of vectors~\cite{LittlewoodOfford1938} (later proved in all dimensions by Kleitman~\cite{Kleitman1970}).

Shearer and Kleitman~\cite{ShearerKleitman1979} were the first to investigate chain decompositions of the $n$-cube that are different from the aforementioned standard decomposition.
They proved that, when picking subsets $x,y\subseteq [n]$ at random, the probability that $x\subseteq y$ is at least~$1/a_n$, for every probability distribution on~$Q_n$.
Their proof introduces the notion of orthogonal chain decompositions.
Formally, two decompositions of~$Q_n$ into $a_n$ (not necessarily symmetric or saturated) chains are called \emph{orthogonal} if every two chains from the two decompositions have at most a single element of~$Q_n$ in common.
For example, Figure~\ref{fig:Q4} shows three pairwise orthogonal chain decompositions into 6~chains in~$Q_4$.
Shearer and Kleitman conjectured that $Q_n$ admits $b_n:=\lfloor n/2\rfloor + 1$ pairwise orthogonal chain decompositions for all~$n\geq 1$.
As a warm-up exercise, we verified their conjecture for~$n\leq 7$ with computer help.
It is easy to check that there are at most $b_n$ pairwise orthogonal decompositions (consider the node degrees in the Hasse diagram around the middle levels).

As a first step towards their conjecture, Shearer and Kleitman established the existence of \emph{two} orthogonal chain decompositions for all~$n\geq 2$.
They proved this by showing that the standard decomposition and its \emph{complement}, obtained by taking the complements of all sets with respect to the full set~$[n]$, are almost-orthogonal.
Formally, we say that two decompositions of~$Q_n$ into $a_n$ symmetric and saturated chains are \emph{almost-orthogonal} if every two chains from the two decompositions have at most a single element of~$Q_n$ in common, with the exception of the two unique chains of size~$n+1$, which are only allowed to intersect in their minimal and maximal elements~$\emptyset$ and~$[n]$.
It is straightforward to verify that for $n\geq 5$, every family of almost-orthogonal decompositions can be modified to orthogonal decompositions, by moving the empty set~$\emptyset$ in all but one of the decompositions from the unique longest chain to a shortest chain, one decomposition at a time (see \cite{ShearerKleitman1979,Spink2017} for details).

Recently, Spink~\cite{Spink2017} made the first progress towards the Shearer-Kleitman conjecture from~1979 by proving that $Q_n$ has \emph{three} pairwise orthogonal chain decompositions for~$n\geq 24$.
He actually showed that $Q_n$ has three almost-orthogonal decompositions into symmetric and saturated chains, from which the result follows as described before.

\subsection{Our results}

Using Spink's product construction, we improve on his result as follows.

\begin{theorem}
\label{thm:ortho}
For all~$n\geq 60$, the $n$-cube has four pairwise almost-orthogonal decompositions 
into symmetric and saturated chains, and consequently four pairwise orthogonal chain decompositions.
\end{theorem}

A slightly weaker notion than almost-orthogonality was introduced in a recent paper by Gregor, J\"ager, M\"utze, Sawada, and Wille~\cite{GregorJMSW18}.
We refer to any cover relation~$x\subseteq y$ as an \emph{edge}~$(x,y)$ ($y$ is one level above $x$), and we say that two decompositions of~$Q_n$ into $a_n$ symmetric and saturated chains are \emph{edge-disjoint} if the two decompositions do not share any edges.
Equivalently, the two decompositions form edge-disjoint paths in the cover graph of~$Q_n$, which is the graph formed by all cover relations.
By this definition, every pair of almost-orthogonal chain decompositions is edge-disjoint, but not necessarily vice versa.
The main application of edge-disjoint chain decompositions in~\cite{GregorJMSW18} was to construct cycle factors in subgraphs of~$Q_n$ induced by an interval of levels around the middle, with the goal of generalizing the recent proof of the middle levels conjecture by M\"utze~\cite{Muetze2016} (see also \cite{GregorMuetzeNummenpalo2018}).
It is also easy to check that $Q_n$ admits at most $b_n$ pairwise edge-disjoint chain decompositions.
The authors of~\cite{GregorJMSW18} conjectured that this bound can be achieved for all~$n\geq 1$.
They verified this conjecture for~$n\leq 7$, and proved that $Q_n$ has \emph{four} pairwise edge-disjoint decompositions for~$n\geq 12$.
We improve on this result as follows.

\begin{theorem}
\label{thm:edge}
For all~$n\geq 90$, the $n$-cube has five pairwise edge-disjoint decompositions into symmetric and saturated chains.
\end{theorem}

Unless stated otherwise, all chains we consider in the following are symmetric and saturated, and we will from now on omit those qualifications.
Moreover, we refer to any decomposition of~$Q_n$ into symmetric and saturated chains as an~\emph{SCD}.
Also, when referring to a family of pairwise almost-orthogonal or pairwise edge-disjoint SCDs, we will from now on omit the qualification `pairwise'.

\subsection{Small dimensions}

Table~\ref{tab:small} summarizes what is known for small values of~$n$.
Specifically, the table shows the maximum number of almost-orthogonal and edge-disjoint SCDs of~$Q_n$ that we know for~$n\leq 25$, together with the upper bound~$b_n$.
As indicated in the table, we actually found six edge-disjoint SCDs of~$Q_{11}$, which, using the product construction from~\cite{GregorJMSW18}, yields six edge-disjoint SCDs for all dimensions $n=11k$, $k \in \NN$.
To extend this result to all but finitely many dimensions, thus improving Theorem~\ref{thm:edge}, we would only need to find six edge-disjoint SCDs of~$Q_n$ for some dimension~$n$ not of this form.
It is also interesting to note that there are \emph{no} three almost-orthogonal SCDs of~$Q_4$ (see~\cite{Spink2017}), i.e., in this case the trivial upper bound~$b_n$ cannot be achieved.
Nevertheless, there are three orthogonal decompositions using non-symmetric chains in~$Q_4$---see Figure~\ref{fig:Q4}---so this shows that not every family of orthogonal chain decompositions can be obtained from almost-orthogonal SCDs.

\begin{table}[htb]
\centering
\newcounter{nodecount}
\newcommand\tabnode[1]{\addtocounter{nodecount}{1}\tikz \node (\arabic{nodecount}) {#1};}
\tikzstyle{every picture}+=[remember picture,baseline]
\tikzstyle{every node}+=[inner sep=0pt,anchor=base,minimum width=0.2cm,align=center,outer sep=4pt]
\def\arraystretch{1.2}
\raggedright
\begin{tabular}{l|lllllllllll}
$n$	& 1 & 2 & 3 & 4 & 5 & 6 & 7 & 8 & 9 & 10 & 11 \\ \hline
almost-orthogonal SCDs &1	&2	&2	&2	&3	&3*	&\hspace{-1mm}\colorbox{gray!50}{4*}	&3*	&3*	&3	&\hspace{-1mm}\colorbox{gray!50}{4*}	\\
edge-disjoint SCDs &1	&2	&2	&3	&3	&4	&4	&\tabnode{4}	&\tabnode{4*}	&\hspace{-1mm}\colorbox{gray!50}{\tabnode{5*}}	&\hspace{-1mm}\colorbox{gray!50}{6*}	\\
upper bound $b_n=\lfloor n/2\rfloor+1$	&1	&2	&2	&3	&3	&4	&4	&\tabnode{5}	&\tabnode{5}	&\tabnode{6\hspace{2mm}}	&6	\\
\end{tabular}$\cdots$ \\[.5cm]
\raggedleft$\cdots$\begin{tabular}{lllllllllllllllllll}
12 & 13 & 14 & 15 & 16 & 17 & 18 & 19 & 20 & 21 & 22 & 23 & 24 & 25 \\ \hline
3  & 3* & 4* & 3  & 3* & 3  & 4* & 3  & 3  & 4* & 4* & 3* & 3  & 4* \\
4  & 4  & 4  & 4  & 4  & 4  & 4  & 4  & 5* & 5* & 6* & 4  & 4  & 4  \\
7  & 7  & 8  & 8  & 9  & 9  & 10 & 10 & 11 & 11 & 12 & 12 & 13 & 13 \\
\end{tabular}
\begin{tikzpicture}[overlay]
\draw [thick,dotted,rounded corners](1.north west) -- (3.north east) -- (6.south east) -- (4.south west) -- cycle;
\end{tikzpicture}
\caption{
Number of almost-orthogonal and edge-disjoint SCDs of~$Q_n$ we know for $n \leq 25$.
Entries marked with * are new compared to the earlier results from~\cite{Spink2017} and \cite{GregorJMSW18}.
For $n\leq 11$, the corresponding families of SCDs are provided electronically on the third authors' website~\cite{www} and on the arXiv~\cite{preprint}, and for the shaded entries they are also shown in Figures~\ref{fig:Q7_4_ortho}--\ref{fig:Q11_6_edge}.
For $n\geq 12$, they are obtained via the product constructions presented in~\cite{Spink2017} and~\cite{GregorJMSW18}.
The entries in the dotted box are explained in Remark~\ref{rem:lbounds}.}
\label{tab:small}
\end{table}

As the table shows, we can also slightly improve Spink's aforementioned result~\cite{Spink2017} that $Q_n$ has three almost-orthogonal SCDs for~$n\geq 24$.
His proof left only the dimensions $n=6,8,9,11,13,16,18,23$ as possible exceptions \cite[Theorem~3.3]{Spink2017} (for $n\geq 5$).
Using the SCDs shown in our table for $n\leq 11$ and Spink's product construction~\cite[Theorem~3.5]{Spink2017}, we can close all those gaps, and obtain that $Q_n$ has three almost-orthogonal SCDs for all $n\geq 5$, and three orthogonal chain decompositions for all $n\geq 4$, providing some more evidence for the Shearer-Kleitman conjecture.
We also see that $Q_n$ has four edge-disjoint SCDs for all $n\geq 6$, ruling out the two possible exceptions $n=9,11$ left by Gregor et al.~\cite{GregorJMSW18}.

\begin{remark}
\label{rem:lbounds}
Our lower bounds for edge-disjoint SCDs differ from the upper bound~$b_n$ by~1 exactly for the dimensions~$n=8,9,10$; see the values in the dotted box in Table~\ref{tab:small}.
In fact, it can be shown that our approach for finding edge-disjoint SCDs via the necklace poset~$N_n$ yields at most~$b_n-1$ edge-disjoint SCDs of~$Q_n$ for all even~$n$ and for~$n=9$ (see Lemma~\ref{lem:cyclic-bound} below), so our methods cannot yield better lower bounds for those cases.
\end{remark}

\subsection{Proof ideas}

We now outline the main ideas for proving Theorems~\ref{thm:ortho} and~\ref{thm:edge}.

\paragraph{Product constructions}

We compute families of~$s=4$ almost-orthogonal and $s=5$ edge-disjoint SCDs, for two cubes $Q_a$ and $Q_b$ of small coprime dimensions~$a$ and $b$.
Specifically, these dimensions will be $(a,b)=(7,11)$ and $(a,b)=(10,11)$, respectively; see the shaded entries in Table~\ref{tab:small}.
Using the product constructions presented in~\cite{Spink2017} and~\cite{GregorJMSW18}, we then obtain $s$~SCDs of the corresponding type for all dimensions~$n$ for which $n$ is a non-negative integer combination of~$a$ and~$b$, in particular for all $n\geq (a-1)(b-1)$.
This evaluates to~$n\geq 60$ and~$n\geq 90$ for the aforementioned pairs~$(a,b)$, respectively.

\paragraph{Problem reduction via the necklace poset}

To find families of SCDs in cubes of small fixed dimension ($n=7$, 10, and~11) that satisfy the desired constraints, we reduce the search space to a much smaller poset, the so-called \emph{necklace poset} $N_n$; see Figure~\ref{fig:NQ5}.
It is obtained from~$Q_n$ by identifying all subsets that differ only in cyclically renaming the elements of the ground set $1\rightarrow 2\rightarrow \cdots\rightarrow n\rightarrow 1$.
The necklace poset~$N_n$ inherits the level structure from~$Q_n$, and notions such as symmetric chains and SCDs translate to it in a natural way.
Moreover, $N_n$ is by a factor of~$n(1-o(1))$ smaller than~$Q_n$, which turns out to be crucial for our computer searches for SCDs.
We refer to the process of translating an SCD computed in~$N_n$ to~$Q_n$ as \emph{unrolling}.
Unrolling essentially creates $n$ copies of each chain in~$N_n$, and these copies are obtained by cyclic renaming as explained before.
This strategy works particularly well when $n$ is a prime number, and with some adjustments it can also be made to work for composite~$n$.
We also introduce a suitable notion of edge multiplicities for the necklace poset (as indicated in Figure~\ref{fig:NQ5}), which allows us to find multiple edge-disjoint SCDs in~$N_n$ simultaneously, and to unroll them to multiple edge-disjoint SCDs in~$Q_n$.
Specifically, we prove that the two constructions of SCDs in~$N_n$ found by Griggs, Killian, and Savage~\cite{GriggsKillianSavage2004} and by Jordan~\cite{Jordan2010} can be unrolled to almost-orthogonal SCDs in~$Q_n$.
The key steps here are Lemmas~\ref{lem:unroll-GKS} and~\ref{lem:unroll-J} and Proposition~\ref{prop:compl} below.

\paragraph{Using SAT solvers}

To search multiple edge-disjoint SCDs in the necklace poset~$N_n$ for some small fixed dimension~$n$, we formulate the problem as a propositional formula in conjunctive normal form (CNF), and compute solutions using the SAT solvers Glucose~\cite{AudemardLS2013} and MiniSat~\cite{EenSorensen2003}.
In our CNF formula, we use Boolean variables that indicate whether certain nodes and edges belong to a particular SCD, and we introduce clauses ensuring that a satisfying variable assignment indeed corresponds to an unrollable SCD, and that multiple SCDs are edge-disjoint.
Once a valid variable assignment is found, we use incremental CNF augmentation to enforce the remaining properties, in particular almost-orthogonality of the unrolled SCDs in~$Q_n$.
Specifically, if we encounter a violation, we add an additional clause that prevents this particular configuration.
We solve the augmented CNF using an incremental SAT solver, until we either find a feasible solution or obtain a formula with no satisfying assignment.
This approach keeps the size of the generated CNFs and of the computation time small, as the solvers can reuse structural information of the CNFs, rather than recomputing a solution from scratch.
The size of the formulas can be reduced further by prescribing some SCDs to be particularly nice decompositions.

\subsection{Related work}

\paragraph{Other chain decompositions}

There is a considerable amount of literature on partitioning the $n$-cube using possibly non-symmetric and/or non-saturated chains.
One of the most interesting open problems in this direction is a well-known conjecture of F\"uredi~\cite{Furedi1985} (cf.~\cite{Griggs1988}), which asserts that $Q_n$ can be decomposed into $a_n$ (not necessarily symmetric or saturated) chains whose sizes differ by at most~1, so their size is~$2^n/a_n$ rounded up or down, which is approximately $\sqrt{\pi n}(1+o(1))$.
Tomon~\cite{Tomon2015} recently made some progress towards this conjecture, by showing that for large enough~$n$, the $n$-cube can be decomposed into $a_n$~chains whose size is between~$0.8\sqrt{n}$ and~$13\sqrt{n}$.
Another remarkable result, recently shown by Gruslys, Leader, and Tomon~\cite{GruslysLeaderTomon2019}, is that for large enough~$n$, the $n$-cube can be partitioned into copies of any fixed poset~$P$, provided that the number of elements of~$P$ is a power of~2 and that $P$ has a unique minimal and maximal element.

Pikurkho~\cite{Pikhurko1999} showed that all edges of the $n$-cube can be partitioned into symmetric chains, but it is not clear whether some of those chains can be selected to form one or more~SCDs.
In a slightly different direction, Streib and Trotter~\cite{StreibTrotter2014} presented the construction of an~SCD of the $n$-cube that can be extended to a Hamiltonian cycle through the entire cover graph.

The existence and construction of SCDs has also been investigated for many graded posets different from~$Q_n$.
The paper~\cite{DeBruijnVETK1951} proves that divisor lattices, which are products of chains, are symmetric chain orders,
and Griggs~\cite{Griggs1977} gave a sufficient condition for a general graded poset to admit an SCD.

Griggs, Killian, and Savage first constructed an explicit SCD of the necklace poset~$N_n$ \cite{GriggsKillianSavage2004} when the dimension~$n$ is a prime number, with the goal of constructing rotation-symmetric Venn diagrams for $n$~curves in the plane (see~\cite{RuskeySavageWagon2006}).
Their result for~$N_n$ with $n$ prime was later generalized by Jordan~\cite{Jordan2010} to all~$n \in \NN$, and to even more general quotients of~$Q_n$ by Duffus, McKibben-Sanders, and Thayer~\cite{DuffusMcKibbenSandersThayer2012}.
All these constructions in the necklace poset proceed by taking suitable subchains from the standard SCD of~$Q_n$.
Further generalizations of these results can be found in \cite{Dhand2012,HershSchilling2013,DuffusThayer2015}.

\paragraph{SAT solvers in combinatorics}

We conclude this section by listing some recent results where SAT solvers were used successfully to tackle difficult problems in (extremal) combinatorics, either by using them to find a solution, or to prove that no solution exists.
Fujita~\cite{Fujita2012} established a new lower bound~$R(4,8) \geq 58$ for the classical Ramsey numbers.
Similarly, Dransfield, Liu, Marek, and Truszczy{\'{n}}ski~\cite{DransfieldLMT2004} derived improved bounds for van der Waerden numbers (see also \cite{HerwigHvLvM2007} and \cite{KourilPaul2008}).
Another recent result that received considerable attention is the paper by Konev and Lisitsa~\cite{KonevLisitsa2014} on the {E}rd\H{o}s discrepancy conjecture.
SAT solvers have also been used in the context of geometry, specifically for tackling Erd\H{o}s-Szekeres type questions, see the papers by Balko and Valtr~\cite{BalkoValtr2017} and by Scheucher~\cite{Scheucher2018}.
Moreover, with their help researchers were able to find new coil-in-the-box Gray codes~\cite{ZinovikKroeningChebiryak2008} and to compute pairs of orthogonal diagonal Latin squares~\cite{ZaikinKochemazovSemenov2016}.



\subsection{Outline of this paper}

In Section~\ref{sec:proofs} we present the proofs of our two main theorems.
The proofs of two crucial lemmas, which settle the base cases for our construction, are deferred to Section~\ref{sec:small} at the end of the paper.
In Section~\ref{sec:unroll} we explain our reduction technique to produce SCDs of~$Q_n$ by working in the much smaller necklace poset~$N_n$, and in Section~\ref{sec:sat} we describe how to exploit this reduction using a SAT solver.

\section{Product constructions implying Theorems~\ref{thm:ortho} and~\ref{thm:edge}}
\label{sec:proofs}

As already mentioned in the introduction, both of our theorems are proved by applying product constructions established in~\cite{Spink2017} and~\cite{GregorJMSW18}, respectively, which allow us to obtain $s$ almost-orthogonal or edge-disjoint SCDs of~$Q_{a+b}$, given $s$ such SCDs in the smaller cubes~$Q_a$ and~$Q_b$.
In the following we will repeatedly use the basic number-theoretic fact that, if~$a$ and~$b$ are coprime integers, then every integer $n\geq (a-1)(b-1)$ is a non-negative integer combination of~$a$ and~$b$.

\subsection{Proof of Theorem~\ref{thm:ortho}}
\label{sec:proof-ortho}

The product construction for almost-orthogonal SCDs requires an additional property that we now define:
A family of almost-orthogonal SCDs of the $n$-cube for some odd~$n$ is called \emph{good} if the union of edges given by all chains of size~2 from all the decompositions forms a unicyclic graph, i.e., a graph all of whose components contain at most a single cycle.
The following crucial statement was proved in~\cite{Spink2017}.

\begin{lemma}[{\cite[Theorem~3.5]{Spink2017}}]
\label{lem:prod-ortho}
Let $s\geq 3$ and $r\geq 2$ be integers, and let $n_1,\dotsc,n_r\geq 3$ be a sequence of odd integers.
If each $Q_{n_i}$, $1\leq i\leq r$, has a good family of~$s$ almost-orthogonal SCDs, then $Q_{n_1+\dotsb+n_r}$ has $s$ almost-orthogonal SCDs.
\end{lemma}

The base case for applying Lemma~\ref{lem:prod-ortho} is the following result, which will be proved in Section~\ref{sec:small}.

\begin{lemma}
\label{lem:ortho}
The cubes~$Q_7$ and~$Q_{11}$ each have four good almost-orthogonal SCDs.
\end{lemma}

\begin{proof}[Proof of Theorem~\ref{thm:ortho}]
As every integer $n\geq (7-1)(11-1)=60$ is a non-negative integer combination of~10 and~11, we can apply Lemmas~\ref{lem:prod-ortho} and~\ref{lem:ortho} to obtain the desired SCDs.
\end{proof}

Spink~\cite[Theorem~3.6]{Spink2017} also proved that the goodness requirement in Lemma~\ref{lem:prod-ortho} can be omitted if the additional condition~$r\geq 6$ is added.
As every integer~$n\geq 60$ is a non-negative integer combination of~7 and~11 with coefficients that sum up to at least~6, we would not need the families of SCDs of~$Q_7$ and~$Q_{11}$ to be good to prove Theorem~\ref{thm:ortho}.
However, since proving this modified version of Lemma~\ref{lem:prod-ortho} is considerably harder, partially deferred to another paper~\cite{DavidSpinkTiba2018}, and since goodness is not hard to achieve on top of almost-orthogonality, we prefer to stick with Lemma~\ref{lem:prod-ortho} in its stated form.
Moreover, in this form the lemma also yields four almost-orthogonal SCDs for all non-negative integer combinations of~7 and~11 that are smaller than~60.

\subsection{Proof of Theorem~\ref{thm:edge}}
\label{sec:proof-edge}

The following product lemma for edge-disjoint SCDs was proved in~\cite{GregorJMSW18}.

\begin{lemma}[{\cite[Theorem~5]{GregorJMSW18}}]
\label{lem:prod-edge}
If~$Q_a$ and~$Q_b$ each have $s$ edge-disjoint SCDs, then $Q_{a+b}$ has $s$ edge-disjoint SCDs.
\end{lemma}

The base case for applying Lemma~\ref{lem:prod-edge} is the following result, which will be proved in Section~\ref{sec:small}.

\begin{lemma}
\label{lem:edge}
The cubes~$Q_{10}$ and~$Q_{11}$ each have five edge-disjoint SCDs.
\end{lemma}

\begin{proof}[Proof of Theorem~\ref{thm:edge}]
As every integer $n\geq (10-1)(11-1)=90$ is a non-negative integer combination of~10 and~11, we can apply Lemmas~\ref{lem:prod-edge} and~\ref{lem:edge} to obtain the desired SCDs.
\end{proof}

To complete the proofs of our main theorems, it remains to prove Lemma~\ref{lem:ortho} and Lemma~\ref{lem:edge}.
The corresponding SCDs are provided in Section~\ref{sec:small} below.

\section{Unrolling the necklace poset}
\label{sec:unroll}

Given a subset~$x\subseteq [n]$, we write~$\sigma(x)$ for the subset obtained from $x$ by cyclically renaming elements $1\rightarrow 2\rightarrow \cdots\rightarrow n\rightarrow 1$.
Moreover, we write~$\neck{x}$ for the family of all subsets obtained by repeatedly applying~$\sigma$ to~$x$, and we refer to $\neck{x}$ as a \emph{necklace}, and to any element of~$\neck{x}$ as a \emph{necklace representative}.
We say that the necklace $\neck{x}$ is \emph{full} if $|\neck{x}|=n$, and \emph{deficient} if $|\neck{x}|<n$.
For example, for $n=4$ the necklace $\neck{\{1,\!3,\!4\}}=\{\{1,\!3,\!4\},\{2,\!4,\!1\},\{3,\!1,\!2\},\{4,\!2,\!3\}\}$ is full, and the necklace $\neck{\{1,\!3\}}=\{\{1,\!3\},\{2,\!4\}\}$ is deficient.
Note that the cardinality of any necklace divides~$n$.
Consequently, if $n$ is a prime number, then~$\neck{\emptyset}$ and~$\neck{[n]}$ are the only deficient necklaces, and all other necklaces are full.
On the other hand, if $n$ is composite, then there are more than these two deficient necklaces.

The \emph{necklace poset~$N_n$} is the set of all necklaces~$\neck{x}$, $x\subseteq [n]$, and its cover relations are all pairs $(\neck{x},\neck{y})$ for which $x\subseteq y$ form a cover relation in the $n$-cube; see the left hand side of Figure~\ref{fig:NQ5}.
Similarly to the $n$-cube, we also refer to the cover relations in~$N_n$ as \emph{edges}.
As $\sigma$ preserves the set size, $N_n$ inherits the level structure from~$Q_n$, and notions such as symmetric chains and SCDs translate to~$N_n$ in the natural way.

As almost all necklaces of~$N_n$ are full, we have that $N_n$ is by a factor of~$n(1-o(1))$ smaller than~$Q_n$, which is vital for our computer searches for SCDs.
We now collect a few simple observations about transferring SCDs from~$N_n$ to~$Q_n$.
These observations are illustrated in Figure~\ref{fig:NQ5}.
Recall that all chains we consider are symmetric and saturated.

\begin{observation}
\label{obs:unroll-full}
Let $y=(y_1,\ldots,y_k)$ be a chain of full necklaces in~$N_n$.
Then there are necklace representatives $x=(x_1,\ldots,x_k)$ with $x_i\in y_i$ for $1\leq i\leq k$, such that $\sigma^i(x)=(\sigma^i(x_1),\ldots,\sigma^i(x_k))$ for $i=0,\ldots,n-1$ is a family of~$n$ disjoint chains in~$Q_n$ that visit exactly all elements from $y_1,\ldots,y_k$.
\end{observation}

The easiest way to pick necklace representatives satisfying those conditions is to move up the chain~$y$ from its minimal element~$y_1$ to its maximal element~$y_k$, starting with an arbitrary representative $x_1\in y_1$, and then arbitrarily picking $x_{j+1}\in y_{j+1}$ for $j=1,\ldots,k-1$ such that $(x_j,x_{j+1})$ is an edge in~$Q_n$.

We refer to the process of translating a chain from~$N_n$ to a family of $n$ chains in~$Q_n$ as described by Observation~\ref{obs:unroll-full} as \emph{unrolling}.
As an example, consider the chain $(y_1,\ldots,y_4)=\big(\neck{\{1\}},\neck{\{1,\!2\}},\neck{\{1,\!2,\!3\}},\neck{\{1,\!2,\!3,\!4\}}\big)$ in~$N_5$.
The necklace representatives $x=(x_1,\ldots,x_4)=(\{1\},\{1,\!2\},\{1,\!2,\!3\},\{1,\!2,\!3,\!4\})$ form a chain in~$Q_5$, and $\sigma^i(x)$, $i=0,\ldots,4$, is a family of five disjoint chains in~$Q_5$ that visit exactly all $5\cdot 4=20$ elements from $y_1,\ldots,y_4$.
It is crucial here to observe that the choice of necklace representatives in Observation~\ref{obs:unroll-full} is \emph{not unique}.
In the previous example, we could also choose $x=(x_1,\ldots,x_4)=(\{1\},\{1,\!5\},\{1,\!4,\!5\},\{1,\!2,\!4,\!5\})$ as necklace representatives, yielding a \emph{different} family of five disjoint chains in~$Q_5$.

The notion of unrolling can be extended straightforwardly from a chain of full necklaces to a chain that has one deficient necklace at each of its ends, as captured by the following observation.
The crucial insight here is that if a necklace~$\neck{x}$ is deficient and of size $d<n$, then $\neck{x}=\{\sigma^i(x)\mid i=0,\ldots,d-1\}$.

\begin{observation}
\label{obs:unroll-def}
Let $(y_0,\ldots,y_{k+1})$ be a chain of necklaces in~$N_n$ such that $y_1,\ldots,y_k$ are full and~$y_0$ and~$y_{k+1}$ are deficient and of the same size~$d<n$.
Then there are necklace representatives $(x_0,\ldots,x_{k+1})$ with $x_i\in y_i$ for $0\leq i\leq k+1$, such that $\sigma^i(x_0,\ldots,x_{k+1})$ for $i=0,\ldots,d-1$, and $\sigma^i(x_1,\ldots,x_k)$ for $i=d,\ldots,n-1$, is a family of $n$ disjoint chains in~$Q_n$ that visit exactly all elements from $y_0,\ldots,y_{k+1}$.
\end{observation}

As an example, consider the chain $y=(y_0,\ldots,y_4)=\big(\neck{\{1,\!5\}},\allowbreak \neck{\{1,\!2,\!5\}},\allowbreak \neck{\{1,\!2,\!3,\!5\}},\allowbreak \neck{\{1,\!2,\!3,\!5,\!6\}},\allowbreak \neck{\{1,\!2,\!3,\!5,\!6,\!7\}}\big)$ in~$N_8$.
It has one deficient necklace of size~$d=4$ at each of its ends, and all inner necklaces are full.
Taking $x=(x_0,\ldots,x_4)=(\{1,\!5\},\allowbreak \{1,\!2,\!5\},\allowbreak \{1,\!2,\!3,\!5\},\allowbreak \{1,\!2,\!3,\!5,\!6\},\allowbreak \{1,\!2,\!3,\!5,\!6,\!7\})$ as necklace representatives, unrolling yields four chains of size~5, namely $\sigma^i(x_0,\ldots,x_4)$ for $i=0,1,2,3$, and it yields four chains of size~3, namely $\sigma^i(x_1,\ldots,x_3)$ for $i=4,5,6,7$.

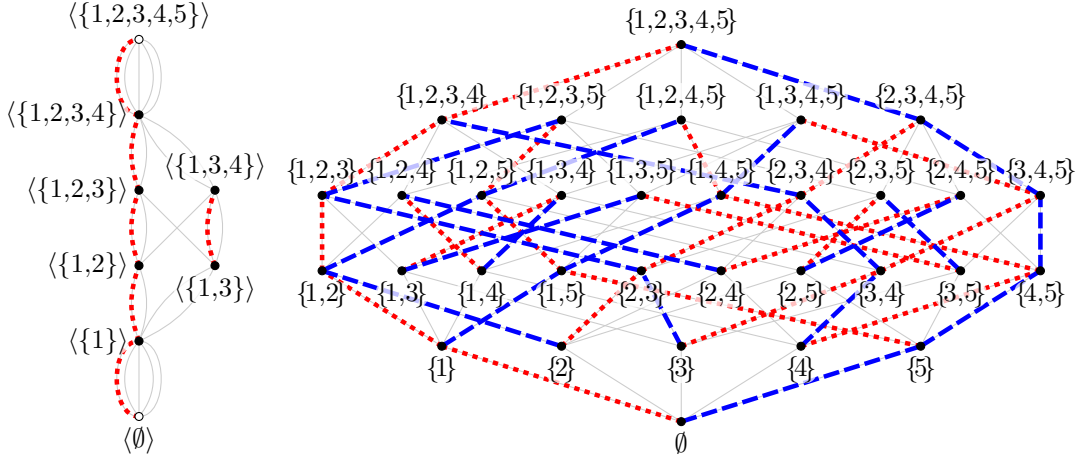
\begin{figure}
\centering
\input{figures/N5_scd.tex}
\input{figures/Q5_unroll.tex}
\caption{Unrolling an SCD of the necklace poset~$N_5$ (left) to an SCD of the 5-cube (right).
The SCD is highlighted by dotted lines, full necklaces are indicated by filled bullets, and deficient necklaces are indicated by empty bullets.
Complementing the resulting SCD of~$Q_5$ yields another SCD (dashed lines), which is edge-disjoint from the first one.
The capacities of the edges of~$N_n$ are visualized by multiple parallel edges.}
\label{fig:NQ5}
\end{figure}

We say that a chain in~$N_n$ is \emph{unimodal} if its minimal and maximal element are necklaces of the same size (possibly deficient), and all other elements are full necklaces.
Moreover, we say that an SCD of~$N_n$ is \emph{unimodal} if all of its chains are unimodal.
Combining Observations~\ref{obs:unroll-full} and~\ref{obs:unroll-def} yields the following fact, which allows us to translate an entire SCD from~$N_n$ to~$Q_n$.

\begin{observation}
\label{obs:unroll-unimodal}
Given a unimodal SCD of~$N_n$, $n\geq 1$, unrolling each of its chains yields an SCD of~$Q_n$.
\end{observation}

This observation is illustrated in Figure~\ref{fig:NQ5}.
We refer to the process of unrolling all chains of an SCD of~$N_n$ to an SCD of~$Q_n$ as \emph{unrolling the SCD}.
Recall that in this unrolling process there may be several choices for picking necklace representatives for each chain.

We now want to simultaneously unroll multiple SCDs from~$N_n$ to edge-disjoint SCDs of~$Q_n$.
This motivates the following definitions:
For any edge $e=(\neck{x},\neck{y})$ of~$N_n$ where~$\neck{x}$ is on level $k\leq (n-1)/2$, we define the \emph{capacity}~$c(e)$ as the number of distinct elements from~$[n]$ that can be added to~$x$ to reach an element in~$\neck{y}$.
For any edge $e=(\neck{y},\neck{x})$ of~$N_n$ where~$\neck{x}$ is on level $k\geq (n+1)/2$, we define the capacity~$c(e)$ symmetrically as the number of distinct elements from~$[n]$ that can be removed from~$x$ to reach an element in~$\neck{y}$.
We can think of the cover graph of~$N_n$ with those capacities on its edges~$e$ as a multigraph with edge multiplicities~$c(e)$; see the left hand side of Figure~\ref{fig:NQ5}.
It is easy to see that the sum of capacities of all edges $e=(\neck{x},\neck{y})$ for fixed $\neck{x}$ on level~$k\leq (n-1)/2$ is~$n-k$, which is equal to the number of neighbors of~$x$ in level~$k+1$ of the cover graph of~$Q_n$.
We say that a family of unimodal SCDs of~$N_n$ is \emph{edge-disjoint} if for every edge~$e$ in~$N_n$, there are at most $c(e)$ chains in those SCDs containing this edge.

For even $n\geq 4$, the middle level of~$N_n$ contains the deficient necklace $\neck{\{1,3,5,\ldots,\allowbreak n-1\}}$.
Consequently, any unimodal chain containing this necklace has size~1.
It follows that the edges incident to this necklace cannot be used by any chain, so that the upper bound~$b_n$ for the maximum number of edge-disjoint SCDs given in the introduction (for $Q_n$) can be improved by~1, yielding the following lemma (see~\cite{Wille2018} for a formal proof).

\begin{lemma}
\label{lem:cyclic-bound}
For even $n\geq 4$, there are at most $b_n-1=n/2$ unimodal SCDs of~$N_n$ that are edge-disjoint.
\end{lemma}

Lemma~\ref{lem:cyclic-bound} shows that our approach via the necklace poset~$N_n$ yields at most four edge-disjoint SCDs of~$N_8$ and at most five edge-disjoint SCDs of~$N_{10}$; recall Remark~\ref{rem:lbounds}.
By considering the deficient necklace~$\neck{\{1,4,7\}}$ and its complement in~$N_9$, one can similarly show that $N_9$ has at most $b_9-1=4$ edge-disjoint SCDs (see~\cite{Wille2018}).

The following lemma was stated and proved in~\cite{GregorJMSW18} in slightly different form.

\begin{lemma}[{\cite[Lemma~7]{GregorJMSW18}}]
\label{lem:prime-unimodal}
Let $n\geq 2$ be a prime number.
Every family of $s\leq b_n$ unimodal SCDs of~$N_n$ that are edge-disjoint can be unrolled to $s$ edge-disjoint SCDs of~$Q_n$.
\end{lemma}

In Section~\ref{sec:small} we will apply Lemma~\ref{lem:prime-unimodal} to prove the case~$n=11$ of Lemma~\ref{lem:edge}.

\begin{figure}
\input{figures/N8.tex}
\hfill
\input{figures/N8_unrolled.tex}
\caption{Given the two unimodal chains in~$N_8$ (left; the chains are dashed and dotted), it is impossible to unroll them so that the resulting sets of chains are edge-disjoint in~$Q_8$ (right).
In each of the four blocks on the right, extending one of the two dashed chains to include the two extreme elements prevents both dotted chains to be extended.
}
\label{fig:not-unrollable}
\end{figure}
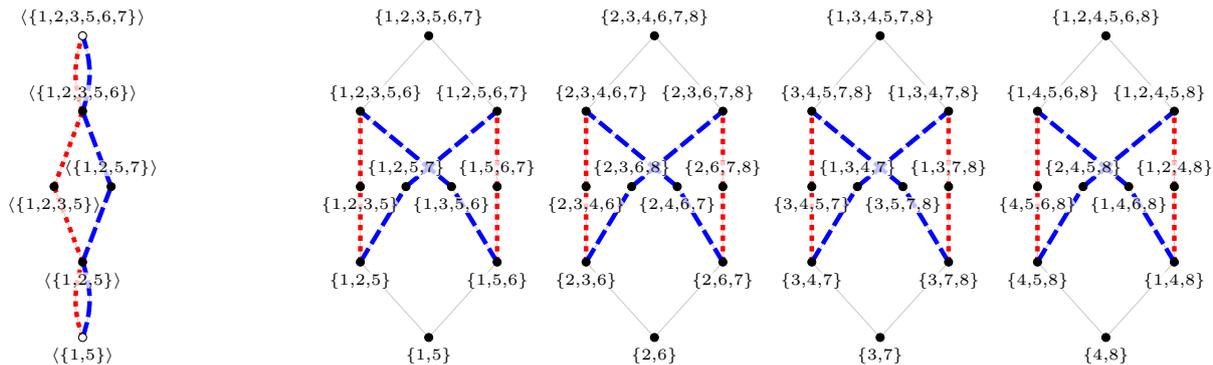

Note that the conclusion of Lemma~\ref{lem:prime-unimodal} does not hold if the dimension~$n$ is not prime, but composite.
The example in Figure~\ref{fig:not-unrollable} shows that even two chains between two deficient necklaces in~$N_n$ cannot always be unrolled so that the resulting sets of chains are edge-disjoint in~$Q_n$.
Consequently, in general it may not be possible to unroll two edge-disjoint SCDs of~$N_n$ to two edge-disjoint SCDs of~$Q_n$.
Nevertheless, the next two lemmas show that unrolling is possible for two known constructions of SCDs of~$N_n$, yielding not only two edge-disjoint SCDs, but even two almost-orthogonal SCDs of~$Q_n$.
Specifically, these constructions are due to Griggs, Killian, and Savage~\cite{GriggsKillianSavage2004} for prime~$n$, and due to Jordan~\cite{Jordan2010} for all~$n$, and they will be explained in the next section.
It is worth to mention that both constructions in general yield different SCDs for prime~$n$; see Figure~\ref{fig:gksj}.

\begin{lemma}
\label{lem:unroll-GKS}
For every prime~$n\geq 2$, the unimodal SCD of~$N_n$ constructed as in~\cite{GriggsKillianSavage2004} and its complement can be unrolled to two almost-orthogonal SCDs of~$Q_n$.
\end{lemma}

\begin{lemma}
\label{lem:unroll-J}
For every $n \geq 1$, the unimodal SCD of~$N_n$ constructed as in~\cite{Jordan2010} and its complement can be unrolled to two almost-orthogonal SCDs of~$Q_n$.
\end{lemma}


In Section~\ref{sec:small} we will apply Lemma~\ref{lem:unroll-GKS} to prove the cases~$n=7$ and~$n=11$ in Lemma~\ref{lem:ortho} and we apply Lemma~\ref{lem:unroll-J} to settle the case~$n=10$ in Lemma~\ref{lem:edge}.
Of course, we could simply check by computer whether these concrete small instances can be unrolled, but we still think that the preceding two lemmas are interesting general facts that have not appeared in the literature before.

The proof of Lemmas~\ref{lem:unroll-GKS} and~\ref{lem:unroll-J} is rather long and technical, and will be given in the next section.
It is followed by Section~\ref{sec:sat}, where we describe our computer search for SCDs of the necklace poset using a SAT solver.
The reader may want to skip these parts for the moment, and continue in Section~\ref{sec:small} with the proofs of Lemmas~\ref{lem:ortho} and \ref{lem:edge}.

\subsection{Proofs of Lemmas~\ref{lem:unroll-GKS} and~\ref{lem:unroll-J}}
\label{sec:proof-unroll}

In the remainder of this section we represent subsets of~$[n]$ by their characteristic $\{0,1\}$-strings of length~$n$.
The $i$th entry of a bitstring~$x$ is denoted by~$x_i$.
For instance, the set $x=\{1,3,5,6\}\subseteq [6]$ is represented by the bitstring $x=x_1\ldots x_6=101011\in\{0,1\}^6$.
The operation~$\sigma(x)$ on the set~$x$ translates to a cyclic right-rotation of the bitstring~$x$.
Moreover, we write $|x|$ for the number of~1s in~$x$, which is the same as the level of~$x$ in~$Q_n$.
Also, for any bitstring~$x$ and any integer~$r\geq 0$, we write $x^r$ for the concatenation of $r$ copies of~$x$.

We begin by recapitulating the SCD constructions in the $n$-cube and the necklace poset described in the three papers~\cite{GreeneKleitman1976,GriggsKillianSavage2004,Jordan2010}.
The first construction by Greene and Kleitman is used as an auxiliary construction for the other two constructions, which we need for proving Lemmas~\ref{lem:unroll-GKS} and~\ref{lem:unroll-J}.
For the reader's convenience, the Greene-Kleitman construction is illustrated in Figure~\ref{fig:paren} for one particular chain, and the other two constructions are illustrated in Figure~\ref{fig:gksj} for~$n=7$.

\paragraph{The Greene-Kleitman construction in~$Q_n$}

Greene and Kleitman~\cite{GreeneKleitman1976} proposed the following explicit construction of an SCD of~$Q_n$.
Given any bitstring~$x$ of length~$n$, we think of every 0-bit as an opening parenthesis, and every 1-bit as a closing parenthesis, and we match closest pairs of opening and closing parentheses in the natural way; see Figure~\ref{fig:paren}.
We let $M(x)$ be the set of all index pairs corresponding to matched parentheses in~$x$, and we let $U_0(x)$ and~$U_1(x)$ be the index sets of unmatched opening and closing parentheses, respectively.
The length of~$x$ clearly satisfies $n=2|M(x)|+|U_0(x)|+|U_1(x)|$.
For any~$x$ with $U_0(x) \neq \emptyset$, we let $\tau(x)$ be the bitstring obtained from $x$ by flipping the leftmost unmatched~0 to a~1.
The union of all chains $(x, \tau(x), \ldots, \tau^k(x))$, where $x\in\{0,1\}^n$ with $U_1(x) = \emptyset$ and $k = |U_0(x)|$, forms an SCD of the $n$-cube for all~$n\geq 1$.
This follows easily from the observation that we have $M(x)=M(\tau(x))=\cdots=M(\tau^k(x))$ and $U_0(\tau^k(x))=\emptyset$ along each such chain, so the chain is uniquely determined by its matched pairs of parentheses.
We denote this SCD by~$D_n$.
This is exactly the standard SCD of the $n$-cube mentioned in the introduction.

\begin{figure}
\centering
\input{figures/paren.tex}
\caption{The parenthesis matching approach for constructing the symmetric chain containing a bitstring~$x\in Q_{22}$.}
\label{fig:paren}
\end{figure}
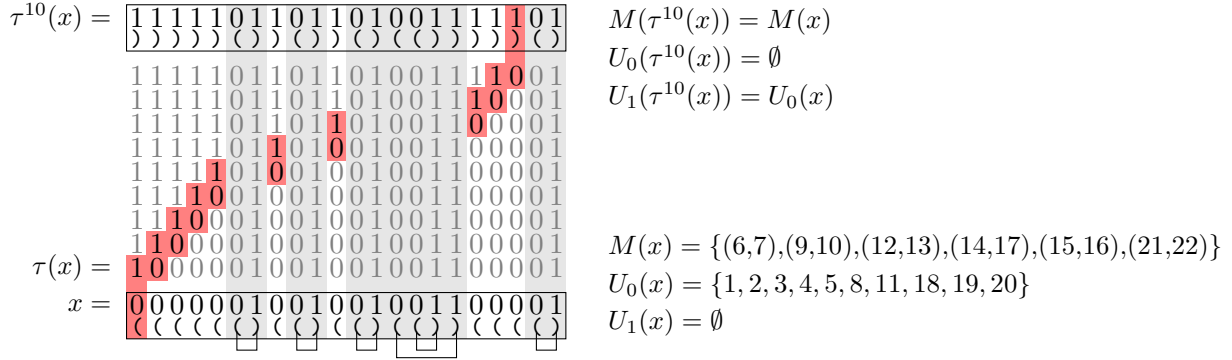

\paragraph{The Griggs-Killian-Savage construction in~$N_n$ for prime~$n$}

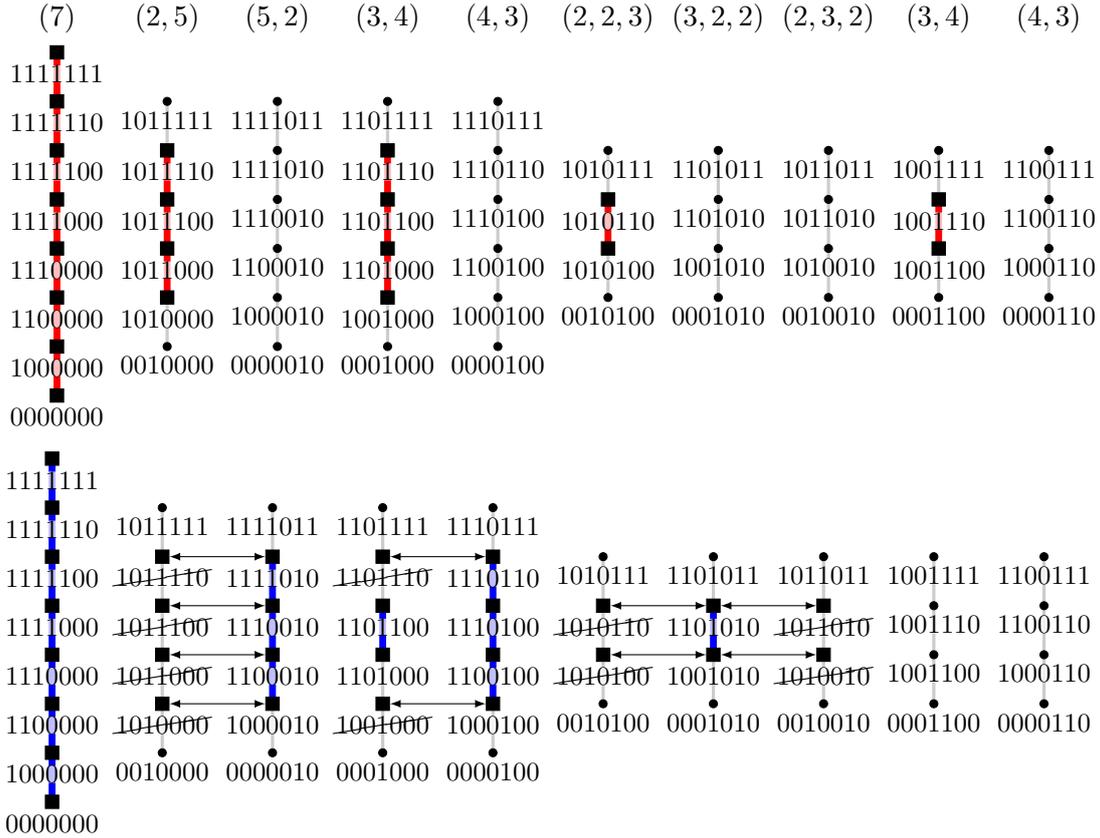
\begin{figure}[t!]
\centering
\input{figures/gks.tex}\\[.8em]
\input{figures/jordan.tex}
\caption{Illustration of the Griggs-Killian-Savage construction (top) and the Jordan construction (bottom) of an SCD of~$N_n$ for $n=7$, by trimming the Greene-Kleitman SCD of~$Q_n$.
The figure shows only the chains from the standard SCD~$D_n$ that contain a bitstring with finite block code, all other chains from~$D_n$ are omitted as they do not contribute to either construction (see Lemmas~\ref{lem:finite-gks} and \ref{lem:finite-jordan}).
The block codes of the inner bitstrings of each chain are shown at the very top; the block codes of all chain endpoints are~$(\infty)$.
Top: We select the bitstrings with lexicographically minimal block code in their necklace (marked by squares) as representatives~$\RGKS_n$.
Bottom: We select all bitstrings with the maximum number of unmatched 1s in their necklace (marked by squares) as representatives~$\RJo_n$.
Representatives of the same necklace with the same number of unmatched 1s are highlighted by double arrows.
If this occurs on two chains of the same length, then the trimming procedure applied in the Jordan construction is not unique.
Note that by the trimming (struck through bitstrings), some chains disappear entirely.
Also observe that the SCDs of~$N_n$ resulting from the two constructions (bold edges) are distinct.
}
\label{fig:gksj}
\end{figure}

To construct an SCD of~$N_n$ for prime~$n$, Griggs, Killian, and Savage~\cite{GriggsKillianSavage2004} use the standard SCD~$D_n$ in~$Q_n$ as a starting point and select subchains of~$D_n$, such that exactly one representative of each necklace is contained in one of the subchains.

For this purpose we define, for any $x\in\{0,1\}^n$, the \emph{block code~$\beta(x)$} as follows:
If $x$ has the form $x = 1^{a_1}0^{b_1}1^{a_2}0^{b_2}\cdots 1^{a_r}0^{b_r}$ with $r\geq 1$ and $a_i,b_i\geq 1$ for all $i=1,\ldots,r$, then $\beta(x) := (a_1 + b_1, a_2 + b_2, \ldots, a_r + b_r)$.
Otherwise we define $\beta(x) := (\infty)$.
If $\beta(x)\neq (\infty)$, then we say that the block code of~$x$ is \emph{finite}.
Note that the block code is finite if and only if $x$ starts with~1 and ends with~0.
Observe also that for all chains from the standard SCD~$D_n$, the block code of all chain endpoints is~$(\infty)$, whereas along the inner bitstrings of each chain, we see the same finite block code along the chain; see Figure~\ref{fig:gksj}.
For prime~$n$, we let $\RGKS_n\subseteq \{0,1\}^n$ be the set of all necklace representatives whose block code is lexicographically minimal in their necklace.
As $n$ is prime, this gives exactly one representative per necklace.\footnote{If $n$ is composite, the method fails, as there may be several necklace representatives with the same block code, e.g., $x=100110$ and $y=110100$ with $\beta(x)=\beta(y)=(3,3)$.}
It was shown in~\cite{GriggsKillianSavage2004} that the representatives $\RGKS_n$ induce symmetric and saturated subchains of~$D_n$, and we denote these subchains by~$\DGKS_n$.
Clearly, the corresponding chains in the necklace poset form an SCD of~$N_n$.

\paragraph{The Jordan construction in~$N_n$ for arbitrary~$n$}

Jordan's construction~\cite{Jordan2010} of an SCD of~$N_n$ for arbitrary~$n$ also uses the standard SCD~$D_n$ in~$Q_n$ as a starting point, but selects subchains in a different fashion.
We let $\RJo_n$ be the set of all necklace representatives that have the maximum number of unmatched~1s in their necklace.
(It is easy to see that those also have finite block code---see Lemma~\ref{lem:finite-jordan}---but this is irrelevant for the moment.)
Note that $\RJo_n$ may contain several representatives from the same necklace; see Figure~\ref{fig:gksj}.
It was shown in~\cite{Jordan2010} that these representatives $\RJo_n$ induce symmetric and saturated subchains of~$D_n$.
We now search for pairs of chains that contain two representatives of the same necklace.
Jordan showed in her paper that these duplicates always lie symmetrically at the ends of both chains, so we may trim the shorter of the two chains symmetrically at both ends.
If both chains have the same size, then we trim any of the two, yielding different resulting subchains.
We repeat this trimming process until each necklace has only a single representative left, and we denote the remaining subchains of~$D_n$ by~$\DJo_n$.
Clearly, the corresponding chains in the necklace poset form an SCD of~$N_n$.
We emphasize again that the outcome of the trimming procedure is not unique, but could be made unique by some lexicographic tie-breaking rule.

\paragraph{Proof of Lemmas~\ref{lem:unroll-GKS} and~\ref{lem:unroll-J}}

The following statements are the main steps for proving Lemmas~\ref{lem:unroll-GKS} and~\ref{lem:unroll-J}.
The proofs of these statements are deferred to the next subsection.

Our first proposition will be used to show that the cyclic rotations of the complement of any subchain of a chain in the standard decomposition~$D_n$ are almost-orthogonal to all other chains in~$D_n$, and it can thus be seen as a generalization of the results of Shearer and Kleitman~\cite{ShearerKleitman1979}.

\begin{proposition}
\label{prop:compl}
Consider two distinct bitstrings~$x$ and $y$ with finite block code that lie on the same chain of~$D_n$.
Then for every $k\geq 0$, the bitstrings $\sigma^{k}(\ol{x})$ and $\sigma^{k}(\ol{y})$ do \emph{not} lie on the same chain of~$D_n$.
\end{proposition}

The next two lemmas capture crucial properties of the subchains~$\DGKS_n$ of~$D_n$ obtained from the Griggs-Killian-Savage construction.

\begin{lemma}
\label{lem:unimod-gks}
For every prime~$n\geq 2$ and every chain from $\DGKS_n$, the corresponding necklaces form a unimodal chain in~$N_n$.
\end{lemma}

\begin{lemma}
\label{lem:finite-gks}
For every prime~$n\geq 2$, all necklace representatives in~$\RGKS_n$ except $0^n$ and $1^n$ have finite block code.
\end{lemma}

The next two lemmas are the analogous statements for the subchains $\DJo_n$ obtained from the Jordan construction.

\begin{lemma}
\label{lem:unimod-jordan}
For every $n\geq 1$ and every chain from $\DJo_n$, the corresponding necklaces form a unimodal chain in~$N_n$.
\end{lemma}

\begin{lemma}
\label{lem:finite-jordan}
For every $n\geq 1$, all necklace representatives in~$\RJo_n$ except $0^n$ and $1^n$ have finite block code.
\end{lemma}

With these lemmas in hand, the proof of Lemmas~\ref{lem:unroll-GKS} and~\ref{lem:unroll-J} is straightforward.

\begin{proof}[Proof of Lemma~\ref{lem:unroll-GKS}]
We first consider the SCD of~$N_n$ for prime $n\geq 2$ obtained via the Griggs-Killian-Savage construction described before, specified by the chains of necklace representatives~$\DGKS_n$.
We let~$U_n$ be the SCD of~$Q_n$ obtained by unrolling each chain from this SCD.
Furthermore, we let~$\ol{U_n}$ be the SCD of~$Q_n$ obtained by unrolling each chain from the complement of this SCD, or equivalently, by taking the complement of~$U_n$.
In both cases, unrolling is possible because of Lemma~\ref{lem:unimod-gks} (recall Observations~\ref{obs:unroll-full} and \ref{obs:unroll-def}), where we also use that complementation preserves unimodality.

It remains to show that $U_n$ and $\ol{U_n}$ are almost-orthogonal SCDs of~$Q_n$.
For this consider two distinct bitstrings~$x'$ and $y'$ on the same chain in~$U_n$ that are neither~$0^n$ nor~$1^n$.
There is a unique~$k\geq 0$, such that $x'=\sigma^k(x)$ and $y'=\sigma^k(y)$ for two bitstrings~$x$ and~$y$ on the same chain in~$\DGKS_n$.
Consider the following chain of implications:
\begin{itemize}[leftmargin=3ex,itemsep=0ex,parsep=0ex,partopsep=0ex,topsep=1mm]
\item $x'$ and $y'$ lie on the same chain in~$\ol{U_n}$.
\item $\ol{x'}$ and $\ol{y'}$ lie on the same chain in~$U_n$.
\item There is a unique~$\ell\geq 0$, so that $\sigma^\ell(\ol{x'})$ and $\sigma^\ell(\ol{y'})$ lie on the same chain in~$\DGKS_n$.
\item $\sigma^{k+\ell}(\ol{x})$ and $\sigma^{k+\ell}(\ol{y})$ lie on the same chain in~$\DGKS_n$.
\end{itemize}
From Lemma~\ref{lem:finite-gks} we know that~$x$ and~$y$ have finite block code.
Clearly, if two elements lie on the same chain in~$\DGKS_n$, then they also lie on the same chain in~$D_n$.
Consequently, applying Proposition~\ref{prop:compl} falsifies the last of the above statements, so the first one is also false, i.e., we obtain that $x'$ and $y'$ do not lie on the same chain in~$\ol{U_n}$.
To complete the proof that~$U_n$ and~$\ol{U_n}$ are almost-orthogonal, we can verify directly that the unique longest chains in~$U_n$ and~$\ol{U_n}$, namely $(0^n,1^10^{n-1},1^20^{n-2},\ldots,1^{n-2}0^2,1^{n-1}0^1,1^n)$ and its complement, intersect only in~$0^n$ and~$1^n$.
\end{proof}

\begin{proof}[Proof of Lemma~\ref{lem:unroll-J}]
This proof proceeds in an analogous fashion as the proof of Lemma~\ref{lem:unroll-GKS} presented before, using Lemmas~\ref{lem:unimod-jordan} and \ref{lem:finite-jordan} instead of Lemmas~\ref{lem:unimod-gks} and \ref{lem:finite-gks}.
\end{proof}

It remains to prove Proposition~\ref{prop:compl} and Lemmas~\ref{lem:unimod-gks}--\ref{lem:finite-jordan}, which will be done in the next three subsections.

\paragraph{Proof of Proposition~\ref{prop:compl}}

For the proof we will need the following auxiliary lemma.

\begin{lemma}
\label{lem:compl-aux}
Let $x,y\in\{0,1\}^n$, and let $i$ and $j$ be two distinct indices such that $x_i x_{i+1}=01$ and $y_i y_{i+1}\neq 01$, and $x_j x_{j+1}\neq 01$ and $y_j y_{j+1}=01$.
Then the sets $M(\sigma^k(x))$ and $M(\sigma^k(y))$ are distinct for all~$k\geq 0$.
\end{lemma}

\begin{proof}
Clearly, for any bitstring $z$ we have that $(\ell,\ell+1)\in M(z)$ if and only if~$z_\ell z_{\ell+1}=01$.
In the following we consider all indices in~$x$ and~$y$ modulo~$n$, with $1,\ldots,n$ as representatives.
As a consequence of our first observation, if $k\neq -i$, then $(k+i,k+i+1)$ is in~$M(\sigma^k(x))$ but not in~$M(\sigma^k(y))$.
Similarly, if $k\neq -j$, then $(k+j,k+j+1)$ is in~$M(\sigma^k(y))$ but not in~$M(\sigma^k(x))$.
As $i\neq j$, the two sets are distinct in any case.
\end{proof}

\begin{proof}[Proof of Proposition~\ref{prop:compl}]
We assume without loss of generality that~$|x|<|y|$, i.e., $y$ is obtained from $x$ by repeatedly applying $\tau$.
Furthermore, let~$i$ and~$j$ be the indices of the leftmost unmatched~0 in~$x$ and the rightmost unmatched~1 in~$y$, respectively.
More formally, we have $i=\min U_0(x)$ and $j=\max U_1(y)$.
As~$x$ and~$y$ have finite block code, we have $x_1=y_1=1$ and $x_n=y_n=0$.
In particular, these positions are unmatched, so $i>1$ and $j<n$ are well-defined.
Moreover, we clearly have~$i\leq j$.
Note that $x_{i-1}$ is either matched or an unmatched~1.
However, as every block of matched parentheses ends with~1, we have $x_{i-1}=y_{i-1}=1$ in any case.
A similar argument shows that $y_{j+1}=x_{j+1}=0$.
Summarizing, the situation looks as follows:
\begin{center}
\input{figures/compl}
\end{center}
From these observations it follows that $\ol{x_{i-1}x_i}=01$ and $\ol{y_{i-1}y_i}=00$, and similarly $\ol{x_j x_{j+1}}=11$ and $\ol{y_j y_{j+1}}=01$.
Applying Lemma~\ref{lem:compl-aux} to the indices~$i-1$ and~$j$ in~$\ol{x}$ and~$\ol{y}$ hence shows that $M(\sigma^k(\ol{x}))$ and $M(\sigma^k(\ol{y}))$ are distinct for all~$k\geq 0$.
As each chain of~$D_n$ is uniquely described by its matched pairs of parentheses, we obtain that $\sigma^k(\ol{x})$ and $\sigma^k(\ol{y})$ do not lie on the same chain, proving the proposition.
\end{proof}

\paragraph{Proofs of Lemmas~\ref{lem:unimod-gks} and \ref{lem:finite-gks}}

\begin{proof}[Proof of Lemma~\ref{lem:unimod-gks}]
This is trivial, as there are only two deficient necklaces for prime~$n$, namely $\neck{0^n}$ and $\neck{1^n}$.
\end{proof}

\begin{proof}[Proof of Lemma~\ref{lem:finite-gks}]
Each bitstring $x$ other than $0^n$ and $1^n$ has two consecutive bits $x_ix_{i+1}=01$.
Consequently, the rotated bitstring $\sigma^{-i}(x)$ starts with~1 and ends with~0 and therefore has finite block code. 
\end{proof}

\paragraph{Proofs of Lemmas~\ref{lem:unimod-jordan} and \ref{lem:finite-jordan}}

\begin{proof}[Proof of Lemma~\ref{lem:unimod-jordan}]
This proof was suggested in Wille's thesis~\cite{Wille2018} on edge-disjoint SCDs in the $n$-cube, and is reproduced here with her permission.

Consider a chain from~$\DJo_n$ with a bitstring~$x$ such that $\neck{x}$ is deficient and of size~$d<n$.
We will show that~$x$ is an endpoint of this chain and that the other endpoint~$y$ corresponds to a deficient necklace~$\neck{y}$ of the same size~$d$.

The following argument is illustrated in Figure~\ref{fig:unimod-jordan}.
Define $r := n/d$ and let $v\in\{0,1\}^d$ be such that $x = v^r$.
We assume without loss of generality that~$|x|\leq n/2$, implying that $|v| \leq d/2$.
As every matched pair of parentheses involves exactly one~0 and one~1, it follows that $|U_0(v)|\geq |U_1(v)|$.
This ensures that we can match every unmatched~1 in the $i$th copy of~$v$ in~$x$ with an unmatched~0 in the $(i-1)$th copy of~$v$ for all $i=2,\ldots,r$, implying that $U_1(x)=U_1(v)$.

\begin{figure}
\centering
\input{figures/unimodal}
\caption{Illustration of the proof of Lemma~\ref{lem:unimod-jordan} for $x = v^r$ with $v=110000010110$ and $r=3$.
The brackets show matched pairs of parentheses, where solid brackets are matches within each copy of~$v$ in~$x$, and dotted brackets are matches across different copies of~$v$.}
\label{fig:unimod-jordan}
\end{figure}

We proceed to show that~$x$ is the starting point of its chain in~$\DJo_n$.
If $U_1(x) = \emptyset$, then~$x$ is the starting point of its chain in~$D_n$ by definition, and consequently also the starting point of its chain in~$\DJo_n$.
Otherwise $U_1(x)\neq \emptyset$, and we show that then $z:=\tau^{-1}(x) \notin \RJo_n$.
By our observation from before, all unmatched~1s of~$x$ lie in the first copy of~$v$, so we have $z=\tau^{-1}(x) = \tau^{-1}(v)v^{r-1}$.
Together with the fact that $|U_1(\tau^{-1}(v))|=|U_1(v)|-1$, we obtain $U_1(z)=U_1(x)-1$ and $U_1(\sigma^{-d}(z))=U_1(x)$.
This implies $U_1(z)<U_1(\sigma^{-d}(z))$, i.e., $z$ does not have the maximum number of unmatched~1s among the representatives of its necklace, so indeed $z \notin \RJo_n$.

We now show that $y:=\tau^{n-2|x|}(x)$ has the form $y = w^r$ for some $w\in\{0,1\}^d$.
This implies $|\neck{y}|\leq d$, which is sufficient to prove that actually $|\neck{y}|=d$, as otherwise we could reverse the roles of~$x$ and~$y$ in the proof, yielding a contradiction.
For $i=1,\ldots,r-1$, the number of~0s in the $i$th copy of~$v$ in~$x$ that are unmatched in~$x$ is $|U_0(v)|-|U_1(v)| = d - 2|v| = (n-2|x|)/r$.
Consequently, by applying $\tau^{n-2|x|}$ to $x$, we arrive at $\tau^{n-2|x|}(x) = w^r$ with $w = \tau^{d-2|v|}(v)$.
\end{proof}

\begin{proof}[Proof of Lemma~\ref{lem:finite-jordan}]
Let $x\in \RJo_n\setminus\{0^n,1^n\}$, i.e., $x$ has the maximum number of unmatched~1s among the representatives of its necklace.
Note that~$x_1=1$, as otherwise we could rotate~$x$ to the left until the first 1-bit reaches the first position, which would strictly increase the number of unmatched~1s.
A similar argument shows that~$x_n=0$, as otherwise we could rotate~$x$ to the right until the last 0-bit reaches the last position, which would strictly increase the number of unmatched~1s.
These two observations imply that $x$ has finite block code.
\end{proof}

\section{SAT based computer search}
\label{sec:sat}

In this section we describe our computer search for SCDs in cubes of small dimension using a SAT solver.

\subsection{The reduced necklace graph}

We let~$N_n^-$ denote the multigraph obtained as follows:
We consider the cover graph of~$N_n$, where the edge multiplicities are given by the capacities (as defined before Lemma~\ref{lem:cyclic-bound}), and we remove all edges between a full necklace and a deficient necklace, whenever the deficient necklace is closer to the middle level(s); see Figure~\ref{fig:N6m}.
Note here that even though $N_n^-$ is a (multi)graph, it inherits the level structure from the poset~$N_n$, so all the poset notions (chain, SCD, etc.) from before translate to~$N_n^-$ in the natural way.
The aforementioned edge removals enforce that a chain containing a deficient necklace must either start or end at this necklace.
Informally speaking, removing those edges does not harm us when searching for unimodal chains and SCDs, as they must not be contained in any unimodal chain anyway.

\subsection{SAT formula for edge-disjoint SCDs of~$N_n^-$}

In this section we describe a propositional formula~$\Phi(n,s)$ in conjunctive normal form (CNF), whose solutions correspond to $s$ edge-disjoint unimodal SCDs of~$N_n^-$.
In the later sections we show how to modify those solutions, so that they can be unrolled to $s$ edge-disjoint (and good almost-orthogonal) SCDs of~$Q_n$.
Throughout this section, the integers~$n\geq 1$ and~$s\geq 2$ are fixed.

We first compute the level sizes of~$N_n^-$, and, based on this, the number~$c_n$ of chains and the chain sizes that an SCD must have.
Different SCDs will be indexed by $i=1,\ldots,s$, and different chains in the $i$th SCD will be indexed by $j=1,\ldots,c_n$.
We also assume that the chains of the $i$th SCD are indexed in decreasing order of their size, so chain~$j=1$ is the unique longest chain, and chain~$j=c_n$ is a shortest chain.

We use Boolean variables $X_{i,j,e}$ to indicate that edge~$e$ of~$N_n^-$ is contained in chain~$j$ of decomposition~$i$.
Moreover, Boolean variables $Y_{i,j,u}$ are used to indicate that node~$u$ of~$N_n^-$ is contained in chain~$j$ of decomposition~$i$.
Clearly, we introduce these variables only for pairs~$(j,e)$ and~$(j,u)$ in the relevant levels.
For instance, the node $u=\neck{\emptyset}$ can only be contained in the longest chain~1, so we only have a single variable~$Y_{i,j,u}$ for fixed~$i$ and~$u$, namely~$Y_{i,1,u}$.

\begin{figure}
\centering
\input{figures/N6m.tex}
\caption{The multigraph $N_6^-$. As before, full necklaces are indicated by filled bullets, and deficient necklaces are indicated by empty bullets.}
\label{fig:N6m}
\end{figure}
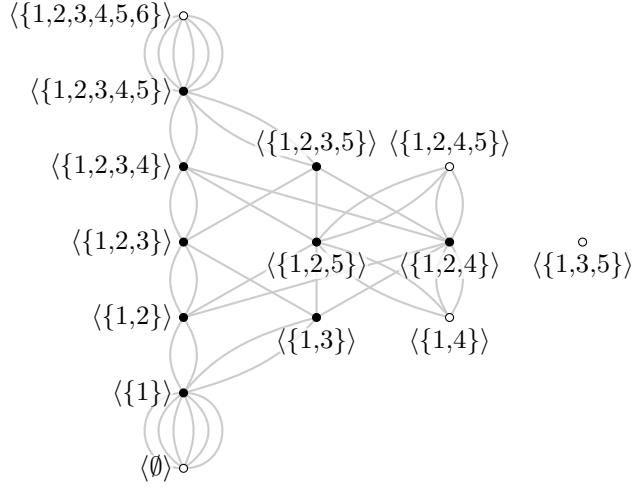

In the following we describe the clauses of our CNF formula~$\Phi(n,s)$ in verbal form.
\begin{description}[leftmargin=3ex]
\item[Link edge to node variables:]
If some edge variable $X_{i,j,e}$ is satisfied and the edge~$e$ connects nodes $u$ and $v$, then both corresponding node variables $Y_{i,j,u}$ and $Y_{i,j,v}$ must be satisfied.
Moreover, if some node variable~$Y_{i,j,u}$ is satisfied and chain~$j$ extends above the level of~$u$, then at least one edge variable~$X_{i,j,e}$ for an edge~$e$ incident with~$u$ and a node from a level above must be satisfied.
Similarly, if chain~$j$ extends below the level of~$u$, then at least one edge variable~$X_{i,j,e}$ for an edge~$e$ incident with~$u$ and a node from a level below must be satisfied.
At a deficient necklace~$u$, one or both of these edge sets are empty, and consequently a chain extending beyond the level of~$u$ in the corresponding direction will never be mapped to~$u$.

\item[Force chains to be present:]
For any chain~$j$ and a level~$k$ visited by this chain, at least one of the node variables~$Y_{i,j,u}$, where $u$ runs over all nodes on level~$k$, must be satisfied.

\item[Node-disjoint chains:] 
For any node~$u$ on a level visited by two chains $j$ and $j'$ in the same SCD~$i$, at most one of the two node variables $Y_{i,j,u}$ or $Y_{i,j',u}$ must be satisfied.

\item[Enforce unimodality:]
For any deficient necklace~$u$ on some level~$k\leq (n-1)/2$, if one of the node variables~$Y_{i,j,u}$ is satisfied, then one of the corresponding variables~$Y_{i,j,v}$, where $v$ is on level~$n-k$ and satisfies $|u|=|v|$, has to be satisfied.
Note that there may be deficient necklaces of different sizes on the same level.

\item[Edge-disjoint SCDs:]
For any two SCDs $i$ and $i'$, any two chains $j$ and $j'$ from those SCDs, and any edge~$e$ between two consecutive levels that are intersected by both chains, at most one of the two edge variables $X_{i,j,e}$ and $X_{i',j',e}$ must be satisfied.
\end{description}

A useful trick to reduce the size of the resulting CNF formula dramatically is to fix some SCDs to be particular standard decompositions, for instance the ones mentioned in Lemmas~\ref{lem:unroll-GKS} and~\ref{lem:unroll-J}, so that the corresponding edge and node variables are not free, but fixed constants.
Similarly, we may also couple certain pairs of SCDs to be complements of each other, so only one set of variables is free, and the other is forced.

\subsection{Unrolling by incremental CNF augmentation}

Any solution of the CNF formula~$\Phi(n,s)$ described before corresponds to $s$ edge-disjoint unimodal SCDs of~$N_n^-$ (and~$N_n$).
However, as the example in Figure~\ref{fig:not-unrollable} shows, these SCDs cannot always be unrolled to $s$ edge-disjoint SCDs of~$Q_n$.
Unfortunately, we have no systematic way to avoid this problem, so we resolve it in an ad-hoc fashion:
We compute a satisfying assignment of~$\Phi(n,s)$ using a SAT solver, and we test whether the current solution can be unrolled.
If not, then we take the first pair of chains from two SCDs that cannot be unrolled simultaneously, and we add an additional clause that prevents this particular pair of chains to appear in a solution, yielding an augmented CNF formula $\Phi'(n,s)$.
The advantage of this approach is that an incremental SAT solver has the ability to reuse information about the structure of~$\Phi(n,s)$ when solving the augmented instance~$\Phi'(n,s)$.
We repeat this iterative process until we either find a solution that can be unrolled to $s$ edge-disjoint SCDs of~$Q_n$, or the resulting CNF formula has no satisfying assignment.
In practice, this last case usually cannot be detected, as the solvers take too long to certify non-satisfiability.

\subsection{Good almost-orthogonal SCDs}

We take a similar incremental approach to compute good families of almost-orthogonal SCDs.
Again we start with the CNF formula~$\Phi(n,s)$, and keep adding constraints that prevent certain pairs of chains to appear.
Specifically, we forbid a pair of chains if it cannot be unrolled or if the unrolled chains intersect in more than one node (or in more than two if these are the longest chains).
It turns out that adding the following clauses right in the beginning speeds up the incremental search process considerably, as it immediately excludes many local violations of almost-orthogonality.
\begin{description}[leftmargin=3ex]
\item[Forbid diamonds:]
Consider four edges~$e=(x,v)$, $f=(v,y)$, $g=(x,w)$, and $h=(w,y)$ that form a `diamond', i.e., $x$ is on some level~$k$, the elements~$v$ and~$w$ are on level~$k+1$, and~$y$ is on level~$k+2$ of~$N_n^-$, such that for any necklace representative of~$x$, flipping the two bits corresponding to~$e$ and~$f$ leads to the same representative of~$y$ as flipping the two bits corresponding to~$g$ and~$h$.
For any two SCDs~$i$ and $i'$ and any two chains~$j$ and~$j'$ from those SCDs that intersect all levels $k$ to $k+2$, not all four edge variables~$X_{i,j,e}$, $X_{i,j,f}$, $X_{i',j',g}$, $X_{i',j',h}$ must be satisfied.
\end{description}
The goodness property could be enforced in a similar incremental way, but coincidentally, the solutions we obtained all satisfied this property right away.
In hindsight, this might not be so surprising, given that the graph formed by the 2-element chains is relatively sparse: Every SCD contributes a matching that satisfies only a $\Theta(1/n)$-fraction of all nodes on average.

\subsection{Implementation details}

We used the incremental SAT solvers Glucose~\cite{AudemardLS2013} and MiniSat~\cite{EenSorensen2003}.
The unrolling tests and incremental CNF augmentation that drive the SAT solver were implemented in C++.
Table~\ref{tab:bench} shows the sizes of the generated SAT instances, running times, and memory requirements for the four families of SCDs that we computed for proving Lemmas~\ref{lem:ortho} and \ref{lem:edge} (shown in Figures~\ref{fig:Q7_4_ortho}--\ref{fig:Q11_6_edge}).

\begin{table}
\centering
\begin{tabular}{rll|rr|lrr}
$n$  & $s$ & SCDs              & \#variables & \#clauses & solver & time & memory\\
\hline
$7$  & $4$ & almost-orthogonal & $1.152$ & $1.484$ & Glucose & 0.048~s & 11~MB \\
$11$ & $4$ & almost-orthogonal & $132.432$ & $1.437.326$ & MiniSat & 23~min & 972~MB \\ 
$10$ & $5$ & edge-disjoint     & $49.900$ & $381.880$ & Glucose & 2:29~h & 501~MB \\
$11$ & $6$ & edge-disjoint     & $198.648$ & $14.258.688$ & Glucose & 3:51~h & 1.6~GB \\
\end{tabular}
\caption{Size of SAT instances and required computing resources.
The number of variables and clauses are recorded at the end of the CNF augmentation and take into account internal simplifications carried out by the solver.}
\label{tab:bench}
\end{table}

\section{Proofs of Lemmas~\ref{lem:ortho} and \ref{lem:edge}}
\label{sec:small}

To prove Lemmas~\ref{lem:ortho} and \ref{lem:edge}, we describe families of four good almost-orthogonal SCDs and five edge-disjoint SCDs of the $n$-cube for $n=7,11$ or $n=10,11$, respectively.
We specify those SCDs in Figures~\ref{fig:Q7_4_ortho}--\ref{fig:Q11_6_edge} in compact form, by unimodal SCDs of the necklace poset~$N_n$, from which the SCDs in the $n$-cube can be recovered by unrolling as described in Section~\ref{sec:unroll} and by taking complements of some of the resulting SCDs.
We specify each chain in one of these SCDs uniquely by a particular choice of necklace representatives (recall Observations~\ref{obs:unroll-full} and \ref{obs:unroll-def} and the remarks between them).
The representatives are described by specifying the minimal and maximal elements of each chain, and the elements from~$[n]$ that are added/removed from the sets when moving along the chains.
The resulting full SCDs of~$Q_n$ are provided in files that can be downloaded from the third authors' website~\cite{www} and on the arXiv~\cite{preprint}, together with a simple Python program for verification.
In those files, subsets of~$[n]$ are encoded by their characteristic bitstrings of length~$n$ (as in Section~\ref{sec:proof-unroll}).

\subsection{Proof of Lemma~\ref{lem:ortho}}

We now describe four good almost-orthogonal SCDs of~$Q_7$ and~$Q_{11}$.
The SCDs $V_7$ and $V_{11}$ defined below are constructed as in~\cite{GriggsKillianSavage2004} (recall Section~\ref{sec:proof-unroll}), and are then unrolled together with their complements as described in Lemma~\ref{lem:unroll-GKS}.

Figure~\ref{fig:Q7_4_ortho} shows two SCDs~$V_7$ and~$W_7$ in~$N_7$, each consisting of 5~chains, that together with their complements $\ol{V_7}$ and $\ol{W_7}$ can be unrolled to four good almost-orthogonal SCDs of~$Q_7$; see the file \texttt{Q7\_4\_ortho.txt}.
Specifically, the union of all $4\cdot 14=56$ edges given by all chains of size~2 of those SCDs forms one cycle of length~14 and 14~paths on 3~edges each.
These are indeed all unicyclic components.

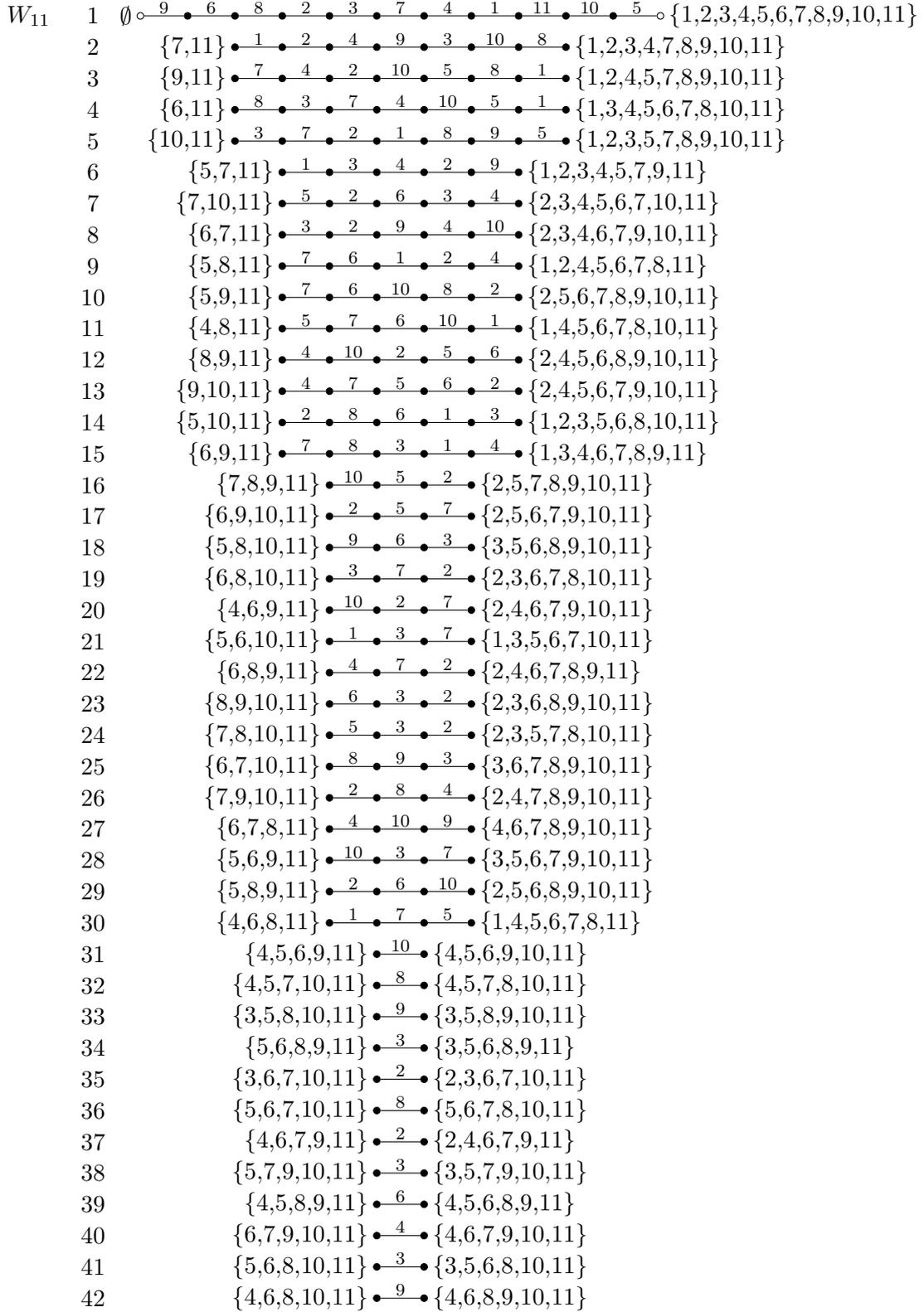
Figure~\ref{fig:Q11_4_ortho} shows two SCDs~$V_{11}$ and~$W_{11}$ in~$N_{11}$, each consisting of 42~chains, that together with their complements $\ol{V_{11}}$ and $\ol{W_{11}}$ can be unrolled to four good almost-orthogonal SCDs of~$Q_{11}$; see the file \texttt{Q11\_4\_ortho.txt}.
Specifically, the union of all $4\cdot 132=528$ edges given by all chains of size~2 of those SCDs forms 66~isolated edges, 22~paths on 2, 3, or 7 edges each, 22~trees on 5~edges (the trees have one degree~3 node with paths of lengths 1, 1, and 3 attached to it), and 2 cycles of length~22 with an additional dangling edge attached to each node.
These are all unicyclic components.

\subsection{Proof of Lemma~\ref{lem:edge}}

We now describe five edge-disjoint SCDs of~$Q_{10}$ and six edge-disjoint SCDs of~$Q_{11}$.
The SCD~$X_{10}$ defined below is constructed as in~\cite{Jordan2010} (recall Section~\ref{sec:proof-unroll}), and is then unrolled together with its complement as described in Lemma~\ref{lem:unroll-J}.
The SCDs of~$Q_{11}$ were computed with the help of Lemma~\ref{lem:prime-unimodal}.

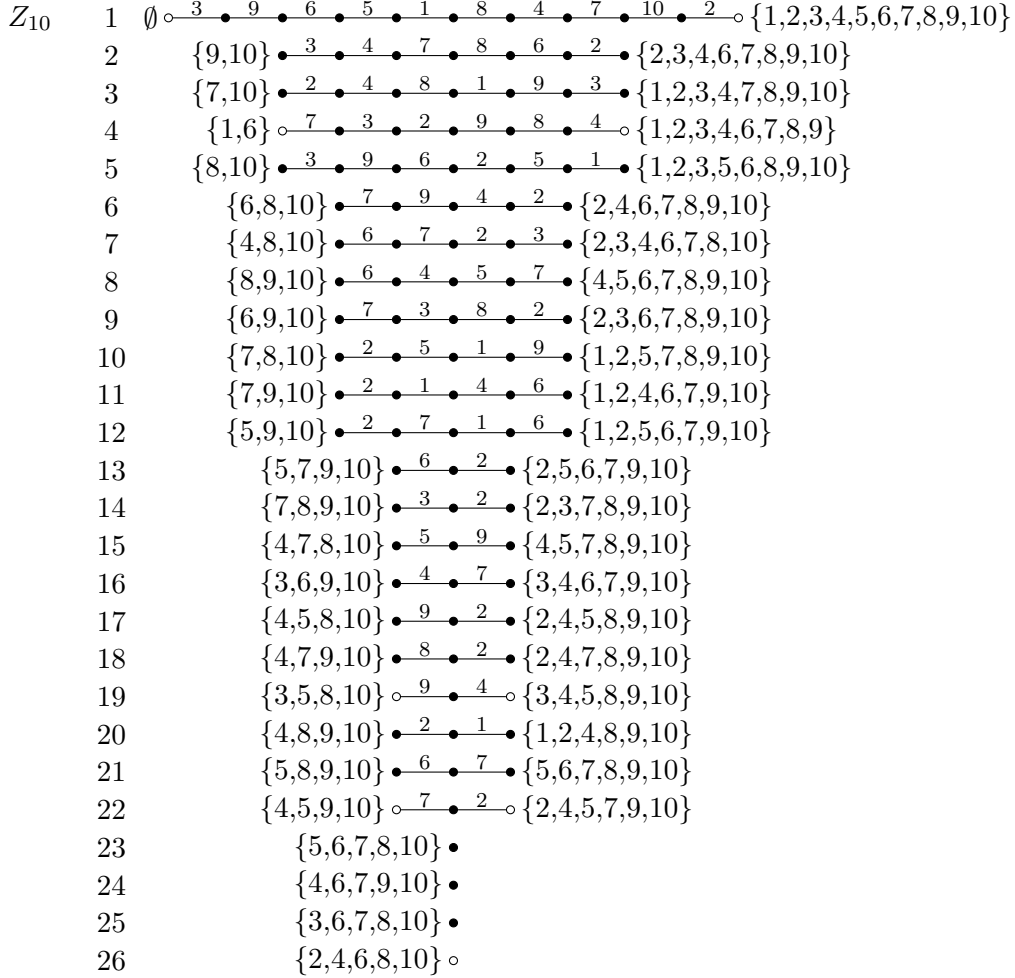
Figure~\ref{fig:Q10_5_edge} shows three SCDs~$X_{10}$, $Y_{10}$, and $Z_{10}$ in~$N_{10}$, each consisting of 26~chains, that together with the complements $\ol{X_{10}}$ and $\ol{Y_{10}}$ can be unrolled to five edge-disjoint SCDs of~$Q_{10}$; see the file \texttt{Q10\_5\_edge.txt}.

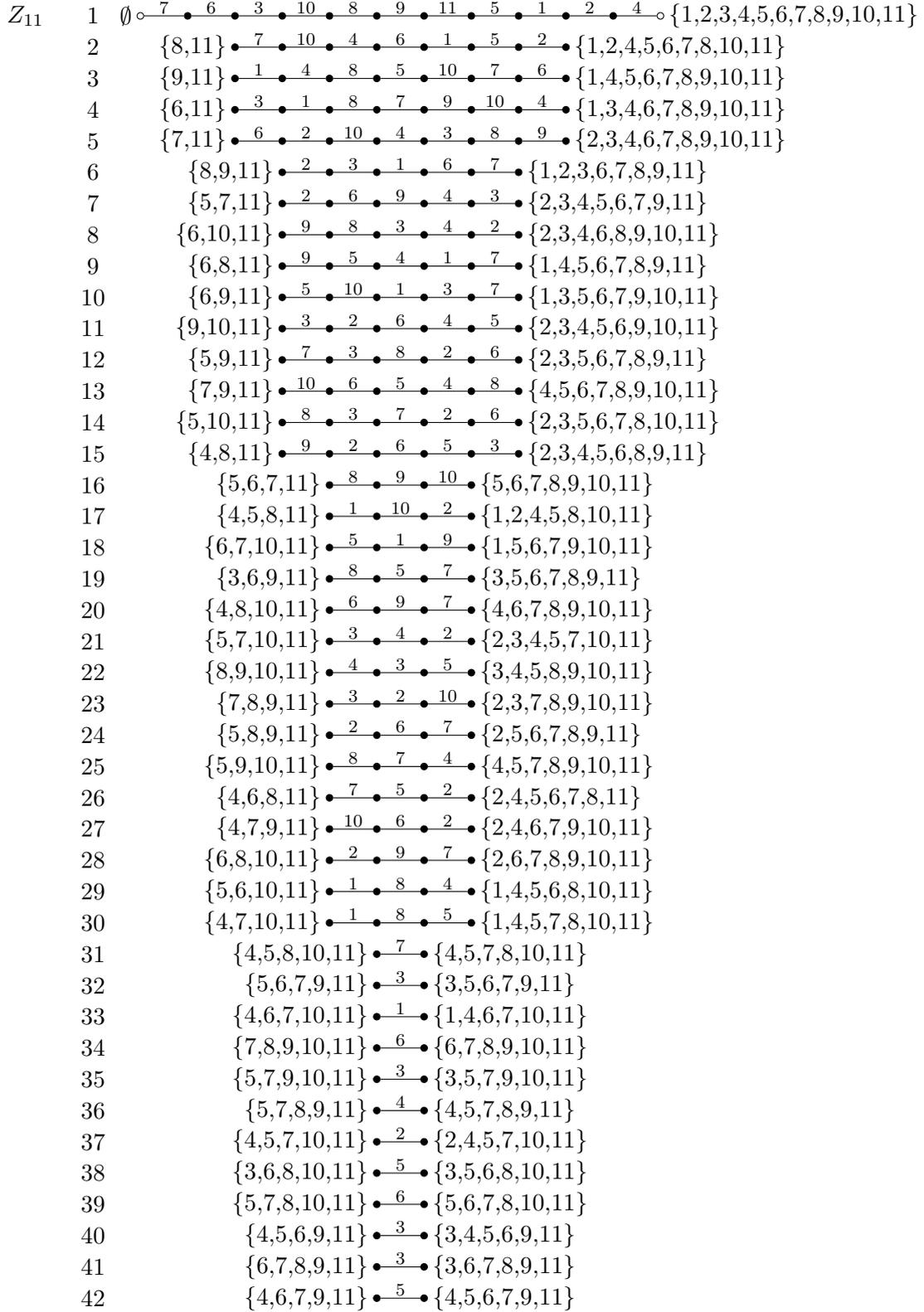
Figure~\ref{fig:Q11_6_edge} shows three SCDs~$X_{11}$, $Y_{11}$, and $Z_{11}$ in~$N_{11}$, each consisting of 42~chains, that together with their complements $\ol{X_{11}}$, $\ol{Y_{11}}$, and $\ol{Z_{11}}$ can be unrolled to six edge-disjoint SCDs of~$Q_{11}$; see the file \texttt{Q11\_6\_edge.txt}.

\vspace{1cm}
\input{figures/VW7.tex}
\input{figures/VW11.tex}
\input{figures/XYZ10.tex}
\input{figures/XYZ11.tex}

\FloatBarrier

\section*{Acknowledgements}

We thank Kaja Wille for several inspiring discussions about symmetric chain decompositions.
We also thank the anonymous referee of this paper whose suggestions improved the presentation in several places.
Manfred Scheucher was supported by DFG Grant FE~340/12-1.

\bibliographystyle{alphaabbrv-url}
\bibliography{refs}

\end{document}

%% file: figures/tikz_styles.tex
\usepackage{tikz}
\usetikzlibrary{calc, arrows, decorations.pathreplacing, matrix, backgrounds}

\def\dotlength{2pt}
\def\dashlength{6pt}
\def\gablength{2pt}

\tikzset{label distance = 2pt,
node_cube/.style={draw = black, fill = black, circle, inner sep = 0pt, text width = 0pt, text height = 0pt, text depth = 0pt, minimum size = 3pt},
node_cube_label/.style={black, rounded corners, inner sep = 0pt, fill = white, text opacity = 1, fill opacity = .75},
node_square/.style={draw = black, fill = black, rectangle, inner sep = 0pt, text width = 0pt, text height = 0pt, text depth = 0pt, minimum size = 5pt},
full/.style={fill=black},
deficient/.style={fill=white},
edge_cube_norm/.style={draw=black!20!white, thin},
edge_chain/.style={draw=black!20!white, very thick},
edge1/.style={
draw=red,line width=0.6mm, dash pattern = {on \dotlength off \gablength}
},
edge2/.style={
draw=blue,line width=0.6mm, dash pattern = {on \dashlength off \gablength}
},
edge3/.style={
draw=green,line width=0.6mm, solid
},
edge4/.style={
draw=brown,line width=0.6mm, dash dot dot
}
}

%% file: figures/Q4.tex
\begin{tikzpicture}
[xscale = 2, yscale = 1.5]

\coordinate (0000) at (0.000000,0);
\coordinate (1000) at (-1.500000,1);
\coordinate (0100) at (-0.500000,1);
\coordinate (1100) at (-2.500000,2);
\coordinate (0010) at (0.500000,1);
\coordinate (1010) at (-1.500000,2);
\coordinate (0110) at (0.500000,2);
\coordinate (1110) at (-1.500000,3);
\coordinate (0001) at (1.500000,1);
\coordinate (1001) at (-0.500000,2);
\coordinate (0101) at (1.500000,2);
\coordinate (1101) at (-0.500000,3);
\coordinate (0011) at (2.500000,2);
\coordinate (1011) at (0.500000,3);
\coordinate (0111) at (1.500000,3);
\coordinate (1111) at (0.000000,4);

\path[edge_cube_norm] (0000) to[] (0001);
\path[edge_cube_norm] (0000) to[] (0010);
\path[edge_cube_norm] (0000) to[] (0100);
\path[edge_cube_norm] (0000) to[] (1000);
\path[edge_cube_norm] (0001) to[] (0011);
\path[edge_cube_norm] (0001) to[] (0101);
\path[edge_cube_norm] (0001) to[] (1001);
\path[edge_cube_norm] (0010) to[] (0011);
\path[edge_cube_norm] (0010) to[] (0110);
\path[edge_cube_norm] (0010) to[] (1010);
\path[edge_cube_norm] (0011) to[] (0111);
\path[edge_cube_norm] (0011) to[] (1011);
\path[edge_cube_norm] (0100) to[] (0101);
\path[edge_cube_norm] (0100) to[] (0110);
\path[edge_cube_norm] (0100) to[] (1100);
\path[edge_cube_norm] (0101) to[] (0111);
\path[edge_cube_norm] (0101) to[] (1101);
\path[edge_cube_norm] (0110) to[] (0111);
\path[edge_cube_norm] (0110) to[] (1110);
\path[edge_cube_norm] (0111) to[] (1111);
\path[edge_cube_norm] (1000) to[] (1001);
\path[edge_cube_norm] (1000) to[] (1010);
\path[edge_cube_norm] (1000) to[] (1100);
\path[edge_cube_norm] (1001) to[] (1011);
\path[edge_cube_norm] (1001) to[] (1101);
\path[edge_cube_norm] (1010) to[] (1011);
\path[edge_cube_norm] (1010) to[] (1110);
\path[edge_cube_norm] (1011) to[] (1111);
\path[edge_cube_norm] (1100) to[] (1101);
\path[edge_cube_norm] (1100) to[] (1110);
\path[edge_cube_norm] (1101) to[] (1111);
\path[edge_cube_norm] (1110) to[] (1111);

\path[edge1] (0000) to[] (1000);
\path[edge1] (1000) to[] (1001);
\path[edge1] (0100) to[] (1100);
\path[edge1] (0010) to[] (1010);
\path[edge1] (1010) to[] (1110);
\path[edge1] (0001) to[] (0101);
\path[edge1] (0101) to[] (1101);
\path[edge1] (0110) to[] (0111);
\path[edge1] (0011) to[] (1011);
\path[edge1] (1011) to[] (1111);
\path[edge2] (0000) to[] (0100);
\path[edge2] (0100) to[] (0101);
\path[edge2] (0010) to[] (0110);
\path[edge2] (1000) to[] (1100);
\path[edge2] (1100) to[] (1110);
\path[edge2] (0001) to[] (0011);
\path[edge2] (0011) to[] (0111);
\path[edge2] (1010) to[] (1011);
\path[edge2] (1001) to[] (1101);
\path[edge2] (1101) to[] (1111);
\path[edge3] (0000) to[] (0010);
\path[edge3] (0010) to[] (0011);
\path[edge3] (1000) to[] (1010);
\path[edge3] (0100) to[] (0110);
\path[edge3] (0110) to[] (1110);
\path[edge3] (0001) to[] (1001);
\path[edge3] (1001) to[] (1011);
\path[edge3] (1100) to[] (1101);
\path[edge3] (0101) to[] (0111);
\path[edge3] (0111) to[] (1111);

\node[node_cube, label={[black, node_cube_label]below:$\emptyset$}] at (0000) {};
\node[node_cube, label={[black, node_cube_label]below:$\{4\}$}] at (0001) {};
\node[node_cube, label={[black, node_cube_label]below:$\{3\}$}] at (0010) {};
\node[node_cube, label={[black, node_cube_label]below:$\{3,\!4\}$}] at (0011) {};
\node[node_cube, label={[black, node_cube_label]below:$\{2\}$}] at (0100) {};
\node[node_cube, label={[black, node_cube_label]below:$\{2,\!4\}$}] at (0101) {};
\node[node_cube, label={[black, node_cube_label]below:$\{2,\!3\}$}] at (0110) {};
\node[node_cube, label={[black, node_cube_label]below:$\{2,\!3,\!4\}$}] at (0111) {};
\node[node_cube, label={[black, node_cube_label]below:$\{1\}$}] at (1000) {};
\node[node_cube, label={[black, node_cube_label]below:$\{1,\!4\}$}] at (1001) {};
\node[node_cube, label={[black, node_cube_label]below:$\{1,\!3\}$}] at (1010) {};
\node[node_cube, label={[black, node_cube_label]below:$\{1,\!3,\!4\}$}] at (1011) {};
\node[node_cube, label={[black, node_cube_label]below:$\{1,\!2\}$}] at (1100) {};
\node[node_cube, label={[black, node_cube_label]below:$\{1,\!2,\!4\}$}] at (1101) {};
\node[node_cube, label={[black, node_cube_label]below:$\{1,\!2,\!3\}$}] at (1110) {};
\node[node_cube, label={[black, node_cube_label]below:$\{1,\!2,\!3,\!4\}$}] at (1111) {};
\end{tikzpicture}

%% file: figures/N5_scd.tex
\begin{tikzpicture}
[xscale = 1, yscale = 1]

\coordinate (00000) at (0.000000,0);
\coordinate (00001) at (0.000000,1);
\coordinate (00011) at (0.000000,2);
\coordinate (00101) at (1.000000,2);
\coordinate (00111) at (0.000000,3);
\coordinate (01011) at (1.000000,3);
\coordinate (01111) at (0.000000,4);
\coordinate (11111) at (0.000000,5);

\path[edge_cube_norm] (11111) to[out = -90.000000, in = 180 - -90.000000, relative] (01111);
\path[edge_cube_norm] (11111) to[out = -40.000000, in = 180 - -40.000000, relative] (01111);
\path[edge_cube_norm] (11111) to[out = 0.000000, in = 180 - 0.000000, relative] (01111);
\path[edge_cube_norm] (11111) to[out = 40.000000, in = 180 - 40.000000, relative] (01111);
\path[edge_cube_norm] (11111) to[out = 90.000000, in = 180 - 90.000000, relative] (01111);
\path[edge_cube_norm] (01111) to[out = -20.000000, in = 180 - -20.000000, relative] (01011);
\path[edge_cube_norm] (01111) to[out = 20.000000, in = 180 - 20.000000, relative] (01011);
\path[edge_cube_norm] (01111) to[out = -20.000000, in = 180 - -20.000000, relative] (00111);
\path[edge_cube_norm] (01111) to[out = 20.000000, in = 180 - 20.000000, relative] (00111);
\path[edge_cube_norm] (00001) to[out = -20.000000, in = 180 - -20.000000, relative] (00101);
\path[edge_cube_norm] (00001) to[out = 20.000000, in = 180 - 20.000000, relative] (00101);
\path[edge_cube_norm] (00001) to[out = -20.000000, in = 180 - -20.000000, relative] (00011);
\path[edge_cube_norm] (00001) to[out = 20.000000, in = 180 - 20.000000, relative] (00011);
\path[edge_cube_norm] (00000) to[out = -90.000000, in = 180 - -90.000000, relative] (00001);
\path[edge_cube_norm] (00000) to[out = -40.000000, in = 180 - -40.000000, relative] (00001);
\path[edge_cube_norm] (00000) to[out = 0.000000, in = 180 - 0.000000, relative] (00001);
\path[edge_cube_norm] (00000) to[out = 40.000000, in = 180 - 40.000000, relative] (00001);
\path[edge_cube_norm] (00000) to[out = 90.000000, in = 180 - 90.000000, relative] (00001);
\path[edge_cube_norm] (00111) to[out = 0.000000, in = 180 - 0.000000, relative] (00101);
\path[edge_cube_norm] (00111) to[out = -20.000000, in = 180 - -20.000000, relative] (00011);
\path[edge_cube_norm] (00111) to[out = 20.000000, in = 180 - 20.000000, relative] (00011);
\path[edge_cube_norm] (01011) to[out = -20.000000, in = 180 - -20.000000, relative] (00101);
\path[edge_cube_norm] (01011) to[out = 20.000000, in = 180 - 20.000000, relative] (00101);
\path[edge_cube_norm] (01011) to[out = 0.000000, in = 180 - 0.000000, relative] (00011);

\path[edge1] (00000) to[bend left=90] (00001);
\path[edge1] (00001) to[bend left=20] (00011);
\path[edge1] (00011) to[bend left=20] (00111);
\path[edge1] (00111) to[bend left=20] (01111);
\path[edge1] (01111) to[bend left=90] (11111);
\path[edge1] (00101) to[bend left=20] (01011);

\node[node_cube, deficient, label={[black, node_cube_label]above:\small$\neck{\{1,\!2,\!3,\!4,\!5\}}$}] at (11111) {};
\node[node_cube, full, label={[black, node_cube_label]left:\small$\neck{\{1,\!2,\!3,\!4\}}$}] at (01111) {};
\node[node_cube, full, label={[black, node_cube_label]left:\small$\neck{\{1\}}$}] at (00001) {};
\node[node_cube, deficient, label={[black, node_cube_label]below:\small$\neck{\emptyset}$}] at (00000) {};
\node[node_cube, full, label={[black, node_cube_label]left:\small$\neck{\{1,\!2,\!3\}}$}] at (00111) {};
\node[node_cube, full, label={[black, node_cube_label]left:\small$\neck{\{1,\!2\}}$}] at (00011) {};
\node[node_cube, full, label={[black, node_cube_label]below:\small$\neck{\{1,\!3\}}$}] at (00101) {};
\node[node_cube, full, label={[black, node_cube_label]above:\small$\neck{\{1,\!3,\!4\}}$}] at (01011) {};
\end{tikzpicture}

%% file: figures/Q5_unroll.tex
\begin{tikzpicture}[xscale = 1.05, yscale = 1]

\def\dist{.5ex}

\coordinate (10000) at (-3.00000,1);
\coordinate (01000) at (-1.50000,1);
\coordinate (00100) at (0.000000,1);
\coordinate (00010) at (1.500000,1);
\coordinate (00001) at (3.000000,1);

\coordinate (11110) at (-3.000000,4);
\coordinate (11101) at (-1.500000,4);
\coordinate (11011) at (0.000000,4);
\coordinate (10111) at (1.500000,4);
\coordinate (01111) at (3.000000,4);

\coordinate (00000) at (0.000000,0);

\coordinate (11000) at (-4.500000,2);
\coordinate (10100) at (-3.500000,2);
\coordinate (10010) at (-2.500000,2);
\coordinate (10001) at (-1.500000,2);
\coordinate (01100) at (-0.500000,2);
\coordinate (01010) at (0.500000,2);
\coordinate (01001) at (1.500000,2);
\coordinate (00110) at (2.500000,2);
\coordinate (00101) at (3.500000,2);
\coordinate (00011) at (4.500000,2);

\coordinate (11100) at (-4.500000,3);
\coordinate (11010) at (-3.500000,3);
\coordinate (11001) at (-2.500000,3);
\coordinate (10110) at (-1.500000,3);
\coordinate (10101) at (-0.500000,3);
\coordinate (10011) at (0.500000,3);
\coordinate (01110) at (1.500000,3);
\coordinate (01101) at (2.500000,3);
\coordinate (01011) at (3.500000,3);
\coordinate (00111) at (4.500000,3);

\coordinate (11111) at (0.000000,5);

\path[edge_cube_norm] (00000) to[] (00001);
\path[edge_cube_norm] (00000) to[] (00010);
\path[edge_cube_norm] (00000) to[] (00100);
\path[edge_cube_norm] (00000) to[] (01000);
\path[edge_cube_norm] (00000) to[] (10000);
\path[edge_cube_norm] (00001) to[] (00011);
\path[edge_cube_norm] (00001) to[] (00101);
\path[edge_cube_norm] (00001) to[] (01001);
\path[edge_cube_norm] (00001) to[] (10001);
\path[edge_cube_norm] (00010) to[] (00011);
\path[edge_cube_norm] (00010) to[] (00110);
\path[edge_cube_norm] (00010) to[] (01010);
\path[edge_cube_norm] (00010) to[] (10010);
\path[edge_cube_norm] (00011) to[] (00111);
\path[edge_cube_norm] (00011) to[] (01011);
\path[edge_cube_norm] (00011) to[] (10011);
\path[edge_cube_norm] (00100) to[] (00101);
\path[edge_cube_norm] (00100) to[] (00110);
\path[edge_cube_norm] (00100) to[] (01100);
\path[edge_cube_norm] (00100) to[] (10100);
\path[edge_cube_norm] (00101) to[] (00111);
\path[edge_cube_norm] (00101) to[] (01101);
\path[edge_cube_norm] (00101) to[] (10101);
\path[edge_cube_norm] (00110) to[] (00111);
\path[edge_cube_norm] (00110) to[] (01110);
\path[edge_cube_norm] (00110) to[] (10110);
\path[edge_cube_norm] (00111) to[] (01111);
\path[edge_cube_norm] (00111) to[] (10111);
\path[edge_cube_norm] (01000) to[] (01001);
\path[edge_cube_norm] (01000) to[] (01010);
\path[edge_cube_norm] (01000) to[] (01100);
\path[edge_cube_norm] (01000) to[] (11000);
\path[edge_cube_norm] (01001) to[] (01011);
\path[edge_cube_norm] (01001) to[] (01101);
\path[edge_cube_norm] (01001) to[] (11001);
\path[edge_cube_norm] (01010) to[] (01011);
\path[edge_cube_norm] (01010) to[] (01110);
\path[edge_cube_norm] (01010) to[] (11010);
\path[edge_cube_norm] (01011) to[] (01111);
\path[edge_cube_norm] (01011) to[] (11011);
\path[edge_cube_norm] (01100) to[] (01101);
\path[edge_cube_norm] (01100) to[] (01110);
\path[edge_cube_norm] (01100) to[] (11100);
\path[edge_cube_norm] (01101) to[] (01111);
\path[edge_cube_norm] (01101) to[] (11101);
\path[edge_cube_norm] (01110) to[] (01111);
\path[edge_cube_norm] (01110) to[] (11110);
\path[edge_cube_norm] (01111) to[] (11111);
\path[edge_cube_norm] (10000) to[] (10001);
\path[edge_cube_norm] (10000) to[] (10010);
\path[edge_cube_norm] (10000) to[] (10100);
\path[edge_cube_norm] (10000) to[] (11000);
\path[edge_cube_norm] (10001) to[] (10011);
\path[edge_cube_norm] (10001) to[] (10101);
\path[edge_cube_norm] (10001) to[] (11001);
\path[edge_cube_norm] (10010) to[] (10011);
\path[edge_cube_norm] (10010) to[] (10110);
\path[edge_cube_norm] (10010) to[] (11010);
\path[edge_cube_norm] (10011) to[] (10111);
\path[edge_cube_norm] (10011) to[] (11011);
\path[edge_cube_norm] (10100) to[] (10101);
\path[edge_cube_norm] (10100) to[] (10110);
\path[edge_cube_norm] (10100) to[] (11100);
\path[edge_cube_norm] (10101) to[] (10111);
\path[edge_cube_norm] (10101) to[] (11101);
\path[edge_cube_norm] (10110) to[] (10111);
\path[edge_cube_norm] (10110) to[] (11110);
\path[edge_cube_norm] (10111) to[] (11111);
\path[edge_cube_norm] (11000) to[] (11001);
\path[edge_cube_norm] (11000) to[] (11010);
\path[edge_cube_norm] (11000) to[] (11100);
\path[edge_cube_norm] (11001) to[] (11011);
\path[edge_cube_norm] (11001) to[] (11101);
\path[edge_cube_norm] (11010) to[] (11011);
\path[edge_cube_norm] (11010) to[] (11110);
\path[edge_cube_norm] (11011) to[] (11111);
\path[edge_cube_norm] (11100) to[] (11101);
\path[edge_cube_norm] (11100) to[] (11110);
\path[edge_cube_norm] (11101) to[] (11111);
\path[edge_cube_norm] (11110) to[] (11111);

\path[edge1] (00000) to[] (10000);
\path[edge1] (10000) to[] (11000);
\path[edge1] (11000) to[] (11100);
\path[edge1] (11100) to[] (11110);
\path[edge1] (11110) to[] (11111);
\path[edge1] (01000) to[] (01100);
\path[edge1] (01100) to[] (01110);
\path[edge1] (01110) to[] (01111);
\path[edge1] (00100) to[] (00110);
\path[edge1] (00110) to[] (00111);
\path[edge1] (00111) to[] (10111);
\path[edge1] (00010) to[] (00011);
\path[edge1] (00011) to[] (10011);
\path[edge1] (10011) to[] (11011);
\path[edge1] (00001) to[] (10001);
\path[edge1] (10001) to[] (11001);
\path[edge1] (11001) to[] (11101);
\path[edge1] (00101) to[] (10101);
\path[edge1] (10010) to[] (11010);
\path[edge1] (01001) to[] (01101);
\path[edge1] (10100) to[] (10110);
\path[edge1] (01010) to[] (01011);
\path[edge2] (00000) to[] (00001);
\path[edge2] (00001) to[] (00011);
\path[edge2] (00011) to[] (00111);
\path[edge2] (00111) to[] (01111);
\path[edge2] (01111) to[] (11111);
\path[edge2] (10000) to[] (10001);
\path[edge2] (10001) to[] (10011);
\path[edge2] (10011) to[] (10111);
\path[edge2] (01000) to[] (11000);
\path[edge2] (11000) to[] (11001);
\path[edge2] (11001) to[] (11011);
\path[edge2] (00100) to[] (01100);
\path[edge2] (01100) to[] (11100);
\path[edge2] (11100) to[] (11101);
\path[edge2] (00010) to[] (00110);
\path[edge2] (00110) to[] (01110);
\path[edge2] (01110) to[] (11110);
\path[edge2] (01010) to[] (11010);
\path[edge2] (00101) to[] (01101);
\path[edge2] (10010) to[] (10110);
\path[edge2] (01001) to[] (01011);
\path[edge2] (10100) to[] (10101);

\node[node_cube, label={[black, node_cube_label]below:\small$\emptyset$}] at (00000) {};
\node[node_cube, label={[black, node_cube_label]below:\small$\{\!5\!\}$}] at (00001) {};
\node[node_cube, label={[black, node_cube_label]below:\small$\{\!4\!\}$}] at (00010) {};
\node[node_cube, label={[black, node_cube_label]below:\small$\{\!4,\!5\!\}$}] at (00011) {};
\node[node_cube, label={[black, node_cube_label]below:\small$\{\!3\!\}$}] at (00100) {};
\node[node_cube, label={[black, node_cube_label]below:\small$\{\!3,\!5\!\}$}] at (00101) {};
\node[node_cube, label={[black, node_cube_label]below:\small$\{\!3,\!4\!\}$}] at (00110) {};
\node[node_cube, label={[black, node_cube_label]above:\small$\{\!3,\!4,\!5\!\}$}] at (00111) {};
\node[node_cube, label={[black, node_cube_label]below:\small$\{\!2\!\}$}] at (01000) {};
\node[node_cube, label={[black, node_cube_label]below:\small$\{\!2,\!5\!\}$}] at (01001) {};
\node[node_cube, label={[black, node_cube_label]below:\small$\{\!2,\!4\!\}$}] at (01010) {};
\node[node_cube, label={[black, node_cube_label]above:\small$\{\!2,\!4,\!5\!\}$}] at (01011) {};
\node[node_cube, label={[black, node_cube_label]below:\small$\{\!2,\!3\!\}$}] at (01100) {};
\node[node_cube, label={[black, node_cube_label]above:\small$\{\!2,\!3,\!5\!\}$}] at (01101) {};
\node[node_cube, label={[black, node_cube_label]above:\small$\{\!2,\!3,\!4\!\}$}] at (01110) {};
\node[node_cube, label={[black, node_cube_label]above:\small$\{\!2,\!3,\!4,\!5\!\}$}] at (01111) {};
\node[node_cube, label={[black, node_cube_label]below:\small$\{\!1\!\}$}] at (10000) {};
\node[node_cube, label={[black, node_cube_label]below:\small$\{\!1,\!5\!\}$}] at (10001) {};
\node[node_cube, label={[black, node_cube_label]below:\small$\{\!1,\!4\!\}$}] at (10010) {};
\node[node_cube, label={[black, node_cube_label]above:\small$\{\!1,\!4,\!5\!\}$}] at (10011) {};
\node[node_cube, label={[black, node_cube_label]below:\small$\{\!1,\!3\!\}$}] at (10100) {};
\node[node_cube, label={[black, node_cube_label]above:\small$\{\!1,\!3,\!5\!\}$}] at (10101) {};
\node[node_cube, label={[black, node_cube_label]above:\small$\{\!1,\!3,\!4\!\}$}] at (10110) {};
\node[node_cube, label={[black, node_cube_label]above:\small$\{\!1,\!3,\!4,\!5\!\}$}] at (10111) {};
\node[node_cube, label={[black, node_cube_label]below:\small$\{\!1,\!2\!\}$}] at (11000) {};
\node[node_cube, label={[black, node_cube_label]above:\small$\{\!1,\!2,\!5\!\}$}] at (11001) {};
\node[node_cube, label={[black, node_cube_label]above:\small$\{\!1,\!2,\!4\!\}$}] at (11010) {};
\node[node_cube, label={[black, node_cube_label]above:\small$\{\!1,\!2,\!4,\!5\!\}$}] at (11011) {};
\node[node_cube, label={[black, node_cube_label]above:\small$\{\!1,\!2,\!3\!\}$}] at (11100) {};
\node[node_cube, label={[black, node_cube_label]above:\small$\{\!1,\!2,\!3,\!5\!\}$}] at (11101) {};
\node[node_cube, label={[black, node_cube_label]above:\small$\{\!1,\!2,\!3,\!4\!\}$}] at (11110) {};
\node[node_cube, label={[black, node_cube_label]above:\small$\{\!1,\!2,\!3,\!4,\!5\!\}$}] at (11111) {};
\end{tikzpicture}

%% file: figures/N8.tex
\begin{tikzpicture}[xscale = .75, yscale = 1]

\def\xa{2}
\def\xb{3}
\def\xc{4}
\def\xd{5}
\def\xe{6}

\coordinate (00010001) at (0,\xa);
\coordinate (00010011) at (0,\xb);
\coordinate (00010111) at (-.5,\xc);
\coordinate (01010011) at (.5,\xc);
\coordinate (00110111) at (0,\xd);
\coordinate (01110111) at (0,\xe);

\path[edge_cube_norm] (00010001) to[bend left = 25] (00010011);
\path[edge_cube_norm] (00010001) to[bend right = 25] (00010011);
\path[edge_cube_norm] (00010011) to[] (00010111);
\path[edge_cube_norm] (00010011) to[] (01010011);
\path[edge_cube_norm] (00110111) to[] (00010111);
\path[edge_cube_norm] (01110111) to[bend left = 25] (00110111);
\path[edge_cube_norm] (01110111) to[bend right = 25] (00110111);

\path[edge1] (00010001) to[bend left = 25] (00010011);
\path[edge1] (00010011) to[] (00010111);
\path[edge1] (00010111) to[] (00110111);
\path[edge1] (00110111) to[bend left = 25] (01110111);
\path[edge2] (00010001) to[bend right = 25] (00010011);
\path[edge2] (00010011) to[] (01010011);
\path[edge2] (01010011) to[] (00110111);
\path[edge2] (00110111) to[bend right = 25] (01110111);

\node[node_cube, deficient, label={[black, node_cube_label]below:\tiny$\neck{\{1,\!5\}}$}] at (00010001) {};
\node[node_cube, full, label={[black, node_cube_label]below:\tiny$\neck{\{1,\!2,\!5\}}$}] at (00010011) {};
\node[node_cube, full, label={[black, node_cube_label]below:\tiny$\neck{\{1,\!2,\!3,\!5\}}$}] at (00010111) {};
\node[node_cube, full, label={[black, node_cube_label]above:\tiny$\neck{\{1,\!2,\!5,\!7\}}$}] at (01010011) {};
\node[node_cube, full, label={[black, node_cube_label]above:\tiny$\neck{\{1,\!2,\!3,\!5,\!6\}}$}] at (00110111) {};
\node[node_cube, deficient, label={[black, node_cube_label]above:\tiny$\neck{\{1,\!2,\!3,\!5,\!6,\!7\}}$}] at (01110111) {};
\end{tikzpicture}

%% file: figures/N8_unrolled.tex
\begin{tikzpicture}[xscale = .9, yscale = 1]

\def\xa{2}
\def\xb{3}
\def\xc{4}
\def\xd{5}
\def\xe{6}

\def\shift{3.3}

\def\labelshifttop{0cm}
\def\labelshiftbot{0cm}
\def\labelshiftone{0cm}
\def\labelshifttwo{0cm}
\def\labelshiftthree{0cm}
\def\labelshiftthree{0cm}
\def\labelshiftfourone{.2cm}
\def\labelshiftfourtwo{-.2cm}

\coordinate (00010001) at (1.000000,\xa);
\coordinate (00010011) at (0.000000,\xb);
\coordinate (00110001) at (2.000000,\xb);
\coordinate (00010111) at (0.000000,\xc);
\coordinate (01110001) at (2.000000,\xc);
\coordinate (01010011) at (0.666666,\xc);
\coordinate (00110101) at (1.333333,\xc);
\coordinate (00110111) at (0.000000,\xd);
\coordinate (01110011) at (2.000000,\xd);
\coordinate (01110111) at (1.000000,\xe);

\begin{scope}[xshift = \shift cm * 1]
\coordinate (00100010) at (1.000000,\xa);
\coordinate (00100110) at (0.000000,\xb);
\coordinate (01100010) at (2.000000,\xb);
\coordinate (00101110) at (0.000000,\xc);
\coordinate (11100010) at (2.000000,\xc);
\coordinate (10100110) at (0.666666,\xc);
\coordinate (01101010) at (1.333333,\xc);
\coordinate (01101110) at (0.000000,\xd);
\coordinate (11100110) at (2.000000,\xd);
\coordinate (11101110) at (1.000000,\xe);
\end{scope}

\begin{scope}[xshift = \shift cm * 2]
\coordinate (01000100) at (1.000000,\xa);
\coordinate (01001100) at (0.000000,\xb);
\coordinate (11000100) at (2.000000,\xb);
\coordinate (01011100) at (0.000000,\xc);
\coordinate (11000101) at (2.000000,\xc);
\coordinate (01001101) at (0.666666,\xc);
\coordinate (11010100) at (1.333333,\xc);
\coordinate (11011100) at (0.000000,\xd);
\coordinate (11001101) at (2.000000,\xd);
\coordinate (11011101) at (1.000000,\xe);
\end{scope}

\begin{scope}[xshift = \shift cm *3]
\coordinate (10001000) at (1.000000,\xa);
\coordinate (10011000) at (0.000000,\xb);
\coordinate (10001001) at (2.000000,\xb);
\coordinate (10111000) at (0.000000,\xc);
\coordinate (10001011) at (2.000000,\xc);
\coordinate (10011010) at (0.666666,\xc);
\coordinate (10101001) at (1.333333,\xc);
\coordinate (10111001) at (0.000000,\xd);
\coordinate (10011011) at (2.000000,\xd);
\coordinate (10111011) at (1.000000,\xe);
\end{scope}

\path[edge_cube_norm] (00010001) to[] (00010011);
\path[edge_cube_norm] (00010001) to[] (00110001);
\path[edge_cube_norm] (00100010) to[] (00100110);
\path[edge_cube_norm] (00100010) to[] (01100010);
\path[edge_cube_norm] (01000100) to[] (01001100);
\path[edge_cube_norm] (01000100) to[] (11000100);
\path[edge_cube_norm] (10001000) to[] (10001001);
\path[edge_cube_norm] (10001000) to[] (10011000);
\path[edge_cube_norm] (00010011) to[] (00010111);
\path[edge_cube_norm] (00010011) to[] (01010011);
\path[edge_cube_norm] (10001001) to[] (10001011);
\path[edge_cube_norm] (10001001) to[] (10101001);
\path[edge_cube_norm] (00110001) to[] (00110101);
\path[edge_cube_norm] (00110001) to[] (01110001);
\path[edge_cube_norm] (00100110) to[] (00101110);
\path[edge_cube_norm] (00100110) to[] (10100110);
\path[edge_cube_norm] (01100010) to[] (01101010);
\path[edge_cube_norm] (01100010) to[] (11100010);
\path[edge_cube_norm] (01001100) to[] (01001101);
\path[edge_cube_norm] (01001100) to[] (01011100);
\path[edge_cube_norm] (11000100) to[] (11000101);
\path[edge_cube_norm] (11000100) to[] (11010100);
\path[edge_cube_norm] (10011000) to[] (10011010);
\path[edge_cube_norm] (10011000) to[] (10111000);
\path[edge_cube_norm] (00110111) to[] (00110101);
\path[edge_cube_norm] (00110111) to[] (00010111);
\path[edge_cube_norm] (10011011) to[] (10011010);
\path[edge_cube_norm] (10011011) to[] (10001011);
\path[edge_cube_norm] (01110011) to[] (01110001);
\path[edge_cube_norm] (01110011) to[] (01010011);
\path[edge_cube_norm] (11001101) to[] (11000101);
\path[edge_cube_norm] (11001101) to[] (01001101);
\path[edge_cube_norm] (10111001) to[] (10111000);
\path[edge_cube_norm] (10111001) to[] (10101001);
\path[edge_cube_norm] (01101110) to[] (01101010);
\path[edge_cube_norm] (01101110) to[] (00101110);
\path[edge_cube_norm] (11100110) to[] (11100010);
\path[edge_cube_norm] (11100110) to[] (10100110);
\path[edge_cube_norm] (11011100) to[] (11010100);
\path[edge_cube_norm] (11011100) to[] (01011100);
\path[edge_cube_norm] (01110111) to[] (01110011);
\path[edge_cube_norm] (01110111) to[] (00110111);
\path[edge_cube_norm] (10111011) to[] (10111001);
\path[edge_cube_norm] (10111011) to[] (10011011);
\path[edge_cube_norm] (11011101) to[] (11011100);
\path[edge_cube_norm] (11011101) to[] (11001101);
\path[edge_cube_norm] (11101110) to[] (11100110);
\path[edge_cube_norm] (11101110) to[] (01101110);

\path[edge1] (00010011) to[] (00010111);
\path[edge1] (00010111) to[] (00110111);
\path[edge1] (00100110) to[] (00101110);
\path[edge1] (00101110) to[] (01101110);
\path[edge1] (01001100) to[] (01011100);
\path[edge1] (01011100) to[] (11011100);
\path[edge1] (10011000) to[] (10111000);
\path[edge1] (10111000) to[] (10111001);
\path[edge1] (00110001) to[] (01110001);
\path[edge1] (01110001) to[] (01110011);
\path[edge1] (01100010) to[] (11100010);
\path[edge1] (11100010) to[] (11100110);
\path[edge1] (11000100) to[] (11000101);
\path[edge1] (11000101) to[] (11001101);
\path[edge1] (10001001) to[] (10001011);
\path[edge1] (10001011) to[] (10011011);

\path[edge2] (00010011) to[] (01010011);
\path[edge2] (01010011) to[] (01110011);
\path[edge2] (00100110) to[] (10100110);
\path[edge2] (10100110) to[] (11100110);
\path[edge2] (01001100) to[] (01001101);
\path[edge2] (01001101) to[] (11001101);
\path[edge2] (10011000) to[] (10011010);
\path[edge2] (10011010) to[] (10011011);
\path[edge2] (00110001) to[] (00110101);
\path[edge2] (00110101) to[] (00110111);
\path[edge2] (01100010) to[] (01101010);
\path[edge2] (01101010) to[] (01101110);
\path[edge2] (11000100) to[] (11010100);
\path[edge2] (11010100) to[] (11011100);
\path[edge2] (10001001) to[] (10101001);
\path[edge2] (10101001) to[] (10111001);

\node[node_cube, label={[black, node_cube_label, xshift = \labelshiftbot]below:\tiny$\{1,\!5\}$}] at (00010001) {};
\node[node_cube, label={[black, node_cube_label, xshift = \labelshiftbot]below:\tiny$\{2,\!6\}$}] at (00100010) {};
\node[node_cube, label={[black, node_cube_label, xshift = \labelshiftbot]below:\tiny$\{3,\!7\}$}] at (01000100) {};
\node[node_cube, label={[black, node_cube_label, xshift = \labelshiftbot]below:\tiny$\{4,\!8\}$}] at (10001000) {};

\node[node_cube, label={[black, node_cube_label, xshift = \labelshiftone]below:\tiny$\{1,\!2,\!5\}$}] at (00010011) {};
\node[node_cube, label={[black, node_cube_label, xshift = \labelshiftone]below:\tiny$\{1,\!4,\!8\}$}] at (10001001) {};
\node[node_cube, label={[black, node_cube_label, xshift = \labelshiftone]below:\tiny$\{1,\!5,\!6\}$}] at (00110001) {};
\node[node_cube, label={[black, node_cube_label, xshift = \labelshiftone]below:\tiny$\{2,\!3,\!6\}$}] at (00100110) {};
\node[node_cube, label={[black, node_cube_label, xshift = \labelshiftone]below:\tiny$\{2,\!6,\!7\}$}] at (01100010) {};
\node[node_cube, label={[black, node_cube_label, xshift = \labelshiftone]below:\tiny$\{3,\!4,\!7\}$}] at (01001100) {};
\node[node_cube, label={[black, node_cube_label, xshift = \labelshiftone]below:\tiny$\{3,\!7,\!8\}$}] at (11000100) {};
\node[node_cube, label={[black, node_cube_label, xshift = \labelshiftone]below:\tiny$\{4,\!5,\!8\}$}] at (10011000) {};

\node[node_cube, label={[black, node_cube_label, xshift = \labelshifttwo]below:\tiny$\{1,\!2,\!3,\!5\}$}] at (00010111) {};
\node[node_cube, label={[black, node_cube_label, xshift = \labelshiftthree]above:\tiny$\{1,\!2,\!4,\!8\}$}] at (10001011) {};
\node[node_cube, label={[black, node_cube_label, xshift = \labelshiftthree]above:\tiny$\{1,\!2,\!5,\!7\}$}] at (01010011) {};
\node[node_cube, label={[black, node_cube_label, xshift = \labelshiftthree]above:\tiny$\{1,\!3,\!4,\!7\}$}] at (01001101) {};
\node[node_cube, label={[black, node_cube_label, xshift = \labelshifttwo]below:\tiny$\{1,\!3,\!5,\!6\}$}] at (00110101) {};
\node[node_cube, label={[black, node_cube_label, xshift = \labelshiftthree]above:\tiny$\{1,\!3,\!7,\!8\}$}] at (11000101) {};
\node[node_cube, label={[black, node_cube_label, xshift = \labelshifttwo]below:\tiny$\{1,\!4,\!6,\!8\}$}] at (10101001) {};
\node[node_cube, label={[black, node_cube_label, xshift = \labelshiftthree]above:\tiny$\{1,\!5,\!6,\!7\}$}] at (01110001) {};
\node[node_cube, label={[black, node_cube_label, xshift = \labelshifttwo]below:\tiny$\{2,\!3,\!4,\!6\}$}] at (00101110) {};
\node[node_cube, label={[black, node_cube_label, xshift = \labelshiftthree]above:\tiny$\{2,\!3,\!6,\!8\}$}] at (10100110) {};
\node[node_cube, label={[black, node_cube_label, xshift = \labelshiftthree]above:\tiny$\{2,\!4,\!5,\!8\}$}] at (10011010) {};
\node[node_cube, label={[black, node_cube_label, xshift = \labelshifttwo]below:\tiny$\{2,\!4,\!6,\!7\}$}] at (01101010) {};
\node[node_cube, label={[black, node_cube_label, xshift = \labelshiftthree]above:\tiny$\{2,\!6,\!7,\!8\}$}] at (11100010) {};
\node[node_cube, label={[black, node_cube_label, xshift = \labelshifttwo]below:\tiny$\{3,\!4,\!5,\!7\}$}] at (01011100) {};
\node[node_cube, label={[black, node_cube_label, xshift = \labelshifttwo]below:\tiny$\{3,\!5,\!7,\!8\}$}] at (11010100) {};
\node[node_cube, label={[black, node_cube_label, xshift = \labelshifttwo]below:\tiny$\{4,\!5,\!6,\!8\}$}] at (10111000) {};

\node[node_cube, label={[black, node_cube_label, xshift = \labelshiftfourone]above:\tiny$\{1,\!2,\!3,\!5,\!6\}$}] at (00110111) {};
\node[node_cube, label={[black, node_cube_label, xshift = \labelshiftfourtwo]above:\tiny$\{1,\!2,\!4,\!5,\!8\}$}] at (10011011) {};
\node[node_cube, label={[black, node_cube_label, xshift = \labelshiftfourtwo]above:\tiny$\{1,\!2,\!5,\!6,\!7\}$}] at (01110011) {};
\node[node_cube, label={[black, node_cube_label, xshift = \labelshiftfourtwo]above:\tiny$\{1,\!3,\!4,\!7,\!8\}$}] at (11001101) {};
\node[node_cube, label={[black, node_cube_label, xshift = \labelshiftfourone]above:\tiny$\{1,\!4,\!5,\!6,\!8\}$}] at (10111001) {};
\node[node_cube, label={[black, node_cube_label, xshift = \labelshiftfourone]above:\tiny$\{2,\!3,\!4,\!6,\!7\}$}] at (01101110) {};
\node[node_cube, label={[black, node_cube_label, xshift = \labelshiftfourtwo]above:\tiny$\{2,\!3,\!6,\!7,\!8\}$}] at (11100110) {};
\node[node_cube, label={[black, node_cube_label, xshift = \labelshiftfourone]above:\tiny$\{3,\!4,\!5,\!7,\!8\}$}] at (11011100) {};

\node[node_cube, label={[black, node_cube_label, xshift = \labelshifttop]above:\tiny$\{1,\!2,\!3,\!5,\!6,\!7\}$}] at (01110111) {};
\node[node_cube, label={[black, node_cube_label, xshift = \labelshifttop]above:\tiny$\{1,\!2,\!4,\!5,\!6,\!8\}$}] at (10111011) {};
\node[node_cube, label={[black, node_cube_label, xshift = \labelshifttop]above:\tiny$\{1,\!3,\!4,\!5,\!7,\!8\}$}] at (11011101) {};
\node[node_cube, label={[black, node_cube_label, xshift = \labelshifttop]above:\tiny$\{2,\!3,\!4,\!6,\!7,\!8\}$}] at (11101110) {};
\end{tikzpicture}

%% file: figures/paren.tex
\begin{tikzpicture}[every node/.style={gray, inner sep=1pt}]
 \pgfsetmatrixcolumnsep{0pt}
 \pgfsetmatrixrowsep{0pt}
 \matrix (m) [matrix of nodes] {
  |[black]|1&|[black]|1&|[black]|1&|[black]|1&|[black]|1&|[black]|0&|[black]|1&|[black]|1&|[black]|0&|[black]|1&|[black]|1&|[black]|0&|[black]|1&|[black]|0&|[black]|0&|[black]|1&|[black]|1&|[black]|1&|[black]|1&|[black]|1&|[black]|0&|[black]|1 \\[-1pt]
  |[black]|\footnotesize\texttt)&|[black]|\footnotesize\texttt)&|[black]|\footnotesize\texttt)&|[black]|\footnotesize\texttt)&|[black]|\footnotesize\texttt)&|[black]|\footnotesize\texttt(&|[black]|\footnotesize\texttt)&|[black]|\footnotesize\texttt)&|[black]|\footnotesize\texttt(&|[black]|\footnotesize\texttt)&|[black]|\footnotesize\texttt)&|[black]|\footnotesize\texttt(&|[black]|\footnotesize\texttt)&|[black]|\footnotesize\texttt(&|[black]|\footnotesize\texttt(&|[black]|\footnotesize\texttt)&|[black]|\footnotesize\texttt)&|[black]|\footnotesize\texttt)&|[black]|\footnotesize\texttt)&|[black]|\footnotesize\texttt)&|[black]|\footnotesize\texttt(&|[black]|\footnotesize\texttt) \\[5pt]
  1&1&1&1&1&0&1&1&0&1&1&0&1&0&0&1&1&1&|[black]|1&|[black]|0&0&1 \\
  1&1&1&1&1&0&1&1&0&1&1&0&1&0&0&1&1&|[black]|1&|[black]|0&0&0&1 \\
  1&1&1&1&1&0&1&1&0&1&|[black]|1&0&1&0&0&1&1&|[black]|0&0&0&0&1 \\
  1&1&1&1&1&0&1&|[black]|1&0&1&|[black]|0&0&1&0&0&1&1&0&0&0&0&1 \\
  1&1&1&1&|[black]|1&0&1&|[black]|0&0&1&0&0&1&0&0&1&1&0&0&0&0&1 \\
  1&1&1&|[black]|1&|[black]|0&0&1&0&0&1&0&0&1&0&0&1&1&0&0&0&0&1 \\
  1&1&|[black]|1&|[black]|0&0&0&1&0&0&1&0&0&1&0&0&1&1&0&0&0&0&1 \\
  1&|[black]|1&|[black]|0&0&0&0&1&0&0&1&0&0&1&0&0&1&1&0&0&0&0&1 \\
  |[black]|1&|[black]|0&0&0&0&0&1&0&0&1&0&0&1&0&0&1&1&0&0&0&0&1 \\[5pt]
  |[black]|0&|[black]|0&|[black]|0&|[black]|0&|[black]|0&|[black]|0&|[black]|1&|[black]|0&|[black]|0&|[black]|1&|[black]|0&|[black]|0&|[black]|1&|[black]|0&|[black]|0&|[black]|1&|[black]|1&|[black]|0&|[black]|0&|[black]|0&|[black]|0&|[black]|1 \\[-1pt]
  |[black]|\footnotesize\texttt(&|[black]|\footnotesize\texttt(&|[black]|\footnotesize\texttt(&|[black]|\footnotesize\texttt(&|[black]|\footnotesize\texttt(&|[black]|\footnotesize\texttt(&|[black]|\footnotesize\texttt)&|[black]|\footnotesize\texttt(&|[black]|\footnotesize\texttt(&|[black]|\footnotesize\texttt)&|[black]|\footnotesize\texttt(&|[black]|\footnotesize\texttt(&|[black]|\footnotesize\texttt)&|[black]|\footnotesize\texttt(&|[black]|\footnotesize\texttt(&|[black]|\footnotesize\texttt)&|[black]|\footnotesize\texttt)&|[black]|\footnotesize\texttt(&|[black]|\footnotesize\texttt(&|[black]|\footnotesize\texttt(&|[black]|\footnotesize\texttt(&|[black]|\footnotesize\texttt) \\
 };
 \draw ($(m-13-6.south)+(0,0.1)$) -- ($(m-13-6.south)-(0,0.15)$) -- ($(m-13-7.south)-(0,0.15)$) -- ($(m-13-7.south)+(0,0.1)$);
 \draw ($(m-13-9.south)+(0,0.1)$) -- ($(m-13-9.south)-(0,0.15)$) -- ($(m-13-10.south)-(0,0.15)$) -- ($(m-13-10.south)+(0,0.1)$);
 \draw ($(m-13-12.south)+(0,0.1)$) -- ($(m-13-12.south)-(0,0.15)$) -- ($(m-13-13.south)-(0,0.15)$) -- ($(m-13-13.south)+(0,0.1)$);
 \draw ($(m-13-15.south)+(0,0.1)$) -- ($(m-13-15.south)-(0,0.15)$) -- ($(m-13-16.south)-(0,0.15)$) -- ($(m-13-16.south)+(0,0.1)$);
 \draw ($(m-13-14.south)+(0,0.1)$) -- ($(m-13-14.south)-(0,0.25)$) -- ($(m-13-17.south)-(0,0.25)$) -- ($(m-13-17.south)+(0,0.1)$); 
 \draw ($(m-13-21.south)+(0,0.1)$) -- ($(m-13-21.south)-(0,0.15)$) -- ($(m-13-22.south)-(0,0.15)$) -- ($(m-13-22.south)+(0,0.1)$);
 \draw (m-12-1.north west) rectangle (m-13-22.south east);
 \draw (m-1-1.north west) rectangle (m-2-22.south east);
 
\begin{scope}[on background layer]
 \fill[red!50] (m-9-3.north west) rectangle (m-10-3.south east);
 \fill[red!50] (m-10-2.north west) rectangle (m-11-2.south east);
 \fill[red!50] (m-11-1.north west) rectangle (m-13-1.south east);
 \fill[red!50] (m-8-4.north west) rectangle (m-9-4.south east);
 \fill[red!50] (m-7-5.north west) rectangle (m-8-5.south east);
 \fill[red!50] (m-6-8.north west) rectangle (m-7-8.south east);
 \fill[red!50] (m-5-11.north west) rectangle (m-6-11.south east);
 \fill[red!50] (m-4-18.north west) rectangle (m-5-18.south east);
 \fill[red!50] (m-3-19.north west) rectangle (m-4-19.south east);
 \fill[red!50] (m-1-20.north west) rectangle (m-3-20.south east);
 \foreach \i in {6,7,9,10,12,13,14,15,16,17,21,22}
  \fill[gray!20] (m-1-\i.north west) rectangle (m-13-\i.south east);
\end{scope}

\node[left,anchor=east,black,xshift=-2mm] at (m-12-1) {\small $x={}$};
\node[left,anchor=east,black,xshift=-2mm] at (m-11-1) {\small $\tau(x)={}$};
\node[left,anchor=east,black,xshift=-2mm] at (m-1-1) {\small $\tau^{10}(x)={}$};
\node[right,black] at (3.4,-1) {\small $M(x)=\{(6,\!7),\!(9,\!10),\!(12,\!13),\!(14,\!17),\!(15,\!16),\!(21,\!22)\}$};
\node[right,black] at (3.4,-1.5) {\small $U_0(x)=\{1,2,3,4,5,8,11,18,19,20\}$};
\node[right,black] at (3.4,-2) {\small $U_1(x)=\emptyset$};
\node[right,black] at (3.4,2) {\small $M(\tau^{10}(x))=M(x)$};
\node[right,black] at (3.4,1.5) {\small $U_0(\tau^{10}(x))=\emptyset$};
\node[right,black] at (3.4,1) {\small $U_1(\tau^{10}(x))=U_0(x)$};
\end{tikzpicture}

%% file: figures/gks.tex
\begin{tikzpicture}[xscale = 1.45,yscale = .65]

\def\xboundaryoffset{.025}
\def\yboundaryoffset{.075}

\def\linew{.9mm}

\coordinate (0000000) at (0,0);
\coordinate (1000000) at (0,1);
\coordinate (1100000) at (0,2);
\coordinate (1110000) at (0,3);
\coordinate (1111000) at (0,4);
\coordinate (1111100) at (0,5);
\coordinate (1111111) at (0,7);
\coordinate (1111110) at (0,6);

\coordinate (0010000) at (1,1);
\coordinate (1010000) at (1,2);
\coordinate (1011000) at (1,3);
\coordinate (1011100) at (1,4);
\coordinate (1011110) at (1,5);
\coordinate (1011111) at (1,6);

\coordinate (0000010) at (2,1);
\coordinate (1000010) at (2,2);
\coordinate (1100010) at (2,3);
\coordinate (1110010) at (2,4);
\coordinate (1111010) at (2,5);
\coordinate (1111011) at (2,6);

\coordinate (0001000) at (3,1);
\coordinate (1001000) at (3,2);
\coordinate (1101000) at (3,3);
\coordinate (1101100) at (3,4);
\coordinate (1101110) at (3,5);
\coordinate (1101111) at (3,6);

\coordinate (0000100) at (4,1);
\coordinate (1000100) at (4,2);
\coordinate (1100100) at (4,3);
\coordinate (1110100) at (4,4);
\coordinate (1110110) at (4,5);
\coordinate (1110111) at (4,6);

\coordinate (0010100) at (5,2);
\coordinate (1010100) at (5,3);
\coordinate (1010110) at (5,4);
\coordinate (1010111) at (5,5);

\coordinate (0001010) at (6,2);
\coordinate (1001010) at (6,3);
\coordinate (1101010) at (6,4);
\coordinate (1101011) at (6,5);

\coordinate (0010010) at (7,2);
\coordinate (1010010) at (7,3);
\coordinate (1011010) at (7,4);
\coordinate (1011011) at (7,5);

\coordinate (0001100) at (8,2);
\coordinate (1001100) at (8,3);
\coordinate (1001110) at (8,4);
\coordinate (1001111) at (8,5);

\coordinate (0000110) at (9,2);
\coordinate (1000110) at (9,3);
\coordinate (1100110) at (9,4);
\coordinate (1100111) at (9,5);


\path[edge_chain] (0000000) to[] (1000000);
\path[edge_chain] (1000000) to[] (1100000);
\path[edge_chain] (1100000) to[] (1110000);
\path[edge_chain] (1110000) to[] (1111000);
\path[edge_chain] (1111000) to[] (1111100);
\path[edge_chain] (1111100) to[] (1111110);
\path[edge_chain] (1111110) to[] (1111111);

\path[edge_chain] (0010000) to[] (1010000);
\path[edge_chain] (1010000) to[] (1011000);
\path[edge_chain] (1011000) to[] (1011100);
\path[edge_chain] (1011100) to[] (1011110);
\path[edge_chain] (1011110) to[] (1011111);

\path[edge_chain] (0000010) to[] (1000010);
\path[edge_chain] (1000010) to[] (1100010);
\path[edge_chain] (1100010) to[] (1110010);
\path[edge_chain] (1110010) to[] (1111010);
\path[edge_chain] (1111010) to[] (1111011);

\path[edge_chain] (0001000) to[] (1001000);
\path[edge_chain] (1001000) to[] (1101000);
\path[edge_chain] (1101000) to[] (1101100);
\path[edge_chain] (1101100) to[] (1101110);
\path[edge_chain] (1101110) to[] (1101111);

\path[edge_chain] (0000100) to[] (1000100);
\path[edge_chain] (1000100) to[] (1100100);
\path[edge_chain] (1100100) to[] (1110100);
\path[edge_chain] (1110100) to[] (1110110);
\path[edge_chain] (1110110) to[] (1110111);

\path[edge_chain] (0010100) to[] (1010100);
\path[edge_chain] (1010100) to[] (1010110);
\path[edge_chain] (1010110) to[] (1010111);

\path[edge_chain] (0001010) to[] (1001010);
\path[edge_chain] (1001010) to[] (1101010);
\path[edge_chain] (1101010) to[] (1101011);

\path[edge_chain] (0010010) to[] (1010010);
\path[edge_chain] (1010010) to[] (1011010);
\path[edge_chain] (1011010) to[] (1011011);

\path[edge_chain] (0001100) to[] (1001100);
\path[edge_chain] (1001100) to[] (1001110);
\path[edge_chain] (1001110) to[] (1001111);

\path[edge_chain] (0000110) to[] (1000110);
\path[edge_chain] (1000110) to[] (1100110);
\path[edge_chain] (1100110) to[] (1100111);

\def\edgeoffset{0pt}


\path[edge1,solid, line width = \linew] (0000000) edge[transform canvas={xshift=\edgeoffset}] (1000000);
\path[edge1,solid, line width = \linew] (1000000) edge[transform canvas={xshift=\edgeoffset}] (1100000);
\path[edge1,solid, line width = \linew] (1100000) edge[transform canvas={xshift=\edgeoffset}] (1110000);
\path[edge1,solid, line width = \linew] (1110000) edge[transform canvas={xshift=\edgeoffset}] (1111000);
\path[edge1,solid, line width = \linew] (1111000) edge[transform canvas={xshift=\edgeoffset}] (1111100);
\path[edge1,solid, line width = \linew] (1111100) edge[transform canvas={xshift=\edgeoffset}] (1111110);
\path[edge1,solid, line width = \linew] (1111110) edge[transform canvas={xshift=\edgeoffset}] (1111111);

\path[edge1,solid, line width = \linew] (1010000) to[] (1011000);
\path[edge1,solid, line width = \linew] (1011000) to[] (1011100);
\path[edge1,solid, line width = \linew] (1011100) to[] (1011110);

\path[edge1,solid, line width = \linew] (1001000) to[] (1101000);
\path[edge1,solid, line width = \linew] (1101000) edge[transform canvas={xshift=\edgeoffset}] (1101100);
\path[edge1,solid, line width = \linew] (1101100) to[] (1101110);

\path[edge1,solid, line width = \linew] (1010100) to[] (1010110);

\path[edge1,solid, line width = \linew] (1001100) to[] (1001110);

\node[node_square, label={[node_cube_label]below:\small0000000}] (n0000000) at (0000000) {};
\node[node_cube, label={[node_cube_label]below:\small0000010}] (n0000010) at (0000010) {};
\node[node_cube, label={[node_cube_label]below:\small0000100}] (n0000100) at (0000100) {};
\node[node_cube, label={[node_cube_label]below:\small0000110}] (n0000110) at (0000110) {};
\node[node_cube, label={[node_cube_label]below:\small0001000}] (n0001000) at (0001000) {};
\node[node_cube, label={[node_cube_label]below:\small0001010}] (n0001010) at (0001010) {};
\node[node_cube, label={[node_cube_label]below:\small0001100}] (n0001100) at (0001100) {};
\node[node_cube, label={[node_cube_label]below:\small0010000}] (n0010000) at (0010000) {};
\node[node_cube, label={[node_cube_label]below:\small0010010}] (n0010010) at (0010010) {};
\node[node_cube, label={[node_cube_label]below:\small0010100}] (n0010100) at (0010100) {};
\node[node_square, label={[node_cube_label]below:\small1000000}] (n1000000) at (1000000) {};
\node[node_cube, label={[node_cube_label]below:\small1000010}] (n1000010) at (1000010) {};
\node[node_cube, label={[node_cube_label]below:\small1000100}] (n1000100) at (1000100) {};
\node[node_cube, label={[node_cube_label]below:\small1000110}] (n1000110) at (1000110) {};
\node[node_square, label={[node_cube_label]below:\small1001000}] (n1001000) at (1001000) {};
\node[node_cube, label={[node_cube_label]below:\small1001010}] (n1001010) at (1001010) {};
\node[node_square, label={[node_cube_label]below:\small1001100}] (n1001100) at (1001100) {};
\node[node_square, label={[node_cube_label]below:\small1001110}] (n1001110) at (1001110) {};
\node[node_cube, label={[node_cube_label]below:\small1001111}] (n1001111) at (1001111) {};
\node[node_square, label={[node_cube_label]below:\small1010000}] (n1010000) at (1010000) {};
\node[node_cube, label={[node_cube_label]below:\small1010010}] (n1010010) at (1010010) {};
\node[node_square, label={[node_cube_label]below:\small1010100}] (n1010100) at (1010100) {};
\node[node_square, label={[node_cube_label]below:\small1010110}] (n1010110) at (1010110) {};
\node[node_cube, label={[node_cube_label]below:\small1010111}] (n1010111) at (1010111) {};
\node[node_square, label={[node_cube_label]below:\small1011000}] (n1011000) at (1011000) {};
\node[node_cube, label={[node_cube_label]below:\small1011010}] (n1011010) at (1011010) {};
\node[node_cube, label={[node_cube_label]below:\small1011011}] (n1011011) at (1011011) {};
\node[node_square, label={[node_cube_label]below:\small1011100}] (n1011100) at (1011100) {};
\node[node_square, label={[node_cube_label]below:\small1011110}] (n1011110) at (1011110) {};
\node[node_cube, label={[node_cube_label]below:\small1011111}] (n1011111) at (1011111) {};
\node[node_square, label={[node_cube_label]below:\small1100000}] (n1100000) at (1100000) {};
\node[node_cube, label={[node_cube_label]below:\small1100010}] (n1100010) at (1100010) {};
\node[node_cube, label={[node_cube_label]below:\small1100100}] (n1100100) at (1100100) {};
\node[node_cube, label={[node_cube_label]below:\small1100110}] (n1100110) at (1100110) {};
\node[node_cube, label={[node_cube_label]below:\small1100111}] (n1100111) at (1100111) {};
\node[node_square, label={[node_cube_label]below:\small1101000}] (n1101000) at (1101000) {};
\node[node_cube, label={[node_cube_label]below:\small1101010}] (n1101010) at (1101010) {};
\node[node_cube, label={[node_cube_label]below:\small1101011}] (n1101011) at (1101011) {};
\node[node_square, label={[node_cube_label]below:\small1101100}] (n1101100) at (1101100) {};
\node[node_square, label={[node_cube_label]below:\small1101110}] (n1101110) at (1101110) {};
\node[node_cube, label={[node_cube_label]below:\small1101111}] (n1101111) at (1101111) {};
\node[node_square, label={[node_cube_label]below:\small1110000}] (n1110000) at (1110000) {};
\node[node_cube, label={[node_cube_label]below:\small1110010}] (n1110010) at (1110010) {};
\node[node_cube, label={[node_cube_label]below:\small1110100}] (n1110100) at (1110100) {};
\node[node_cube, label={[node_cube_label]below:\small1110110}] (n1110110) at (1110110) {};
\node[node_cube, label={[node_cube_label]below:\small1110111}] (n1110111) at (1110111) {};
\node[node_square, label={[node_cube_label]below:\small1111000}] (n1111000) at (1111000) {};
\node[node_cube, label={[node_cube_label]below:\small1111010}] (n1111010) at (1111010) {};
\node[node_cube, label={[node_cube_label]below:\small1111011}] (n1111011) at (1111011) {};
\node[node_square, label={[node_cube_label]below:\small1111100}] (n1111100) at (1111100) {};
\node[node_square, label={[node_cube_label]below:\small1111110}] (n1111110) at (1111110) {};
\node[node_square, label={[node_cube_label]below:\small1111111}] (n1111111) at (1111111) {};

\node[inner sep = 0pt] at (0,7.7) {$(7)$};
\node[inner sep = 0pt] at (1,7.7) {$(2,5)$};
\node[inner sep = 0pt] at (2,7.7) {$(5,2)$};
\node[inner sep = 0pt] at (3,7.7) {$(3,4)$};
\node[inner sep = 0pt] at (4,7.7) {$(4,3)$};
\node[inner sep = 0pt] at (5,7.7) {$(2,2,3)$};
\node[inner sep = 0pt] at (6,7.7) {$(3,2,2)$};
\node[inner sep = 0pt] at (7,7.7) {$(2,3,2)$};
\node[inner sep = 0pt] at (8,7.7) {$(3,4)$};
\node[inner sep = 0pt] at (9,7.7) {$(4,3)$};
\end{tikzpicture}

%% file: figures/jordan.tex
\begin{tikzpicture}
[xscale = 1.45,yscale = .65,
equiv/.style={draw, {latex}-{latex}}]

\def\xboundaryoffset{.025}
\def\yboundaryoffset{.075}

\def\linew{.9mm}

\coordinate (0000000) at (0,0);
\coordinate (1000000) at (0,1);
\coordinate (1100000) at (0,2);
\coordinate (1110000) at (0,3);
\coordinate (1111000) at (0,4);
\coordinate (1111100) at (0,5);
\coordinate (1111111) at (0,7);
\coordinate (1111110) at (0,6);

\coordinate (0010000) at (1,1);
\coordinate (1010000) at (1,2);
\coordinate (1011000) at (1,3);
\coordinate (1011100) at (1,4);
\coordinate (1011110) at (1,5);
\coordinate (1011111) at (1,6);

\coordinate (0000010) at (2,1);
\coordinate (1000010) at (2,2);
\coordinate (1100010) at (2,3);
\coordinate (1110010) at (2,4);
\coordinate (1111010) at (2,5);
\coordinate (1111011) at (2,6);

\coordinate (0001000) at (3,1);
\coordinate (1001000) at (3,2);
\coordinate (1101000) at (3,3);
\coordinate (1101100) at (3,4);
\coordinate (1101110) at (3,5);
\coordinate (1101111) at (3,6);

\coordinate (0000100) at (4,1);
\coordinate (1000100) at (4,2);
\coordinate (1100100) at (4,3);
\coordinate (1110100) at (4,4);
\coordinate (1110110) at (4,5);
\coordinate (1110111) at (4,6);

\coordinate (0010100) at (5,2);
\coordinate (1010100) at (5,3);
\coordinate (1010110) at (5,4);
\coordinate (1010111) at (5,5);

\coordinate (0001010) at (6,2);
\coordinate (1001010) at (6,3);
\coordinate (1101010) at (6,4);
\coordinate (1101011) at (6,5);

\coordinate (0010010) at (7,2);
\coordinate (1010010) at (7,3);
\coordinate (1011010) at (7,4);
\coordinate (1011011) at (7,5);

\coordinate (0001100) at (8,2);
\coordinate (1001100) at (8,3);
\coordinate (1001110) at (8,4);
\coordinate (1001111) at (8,5);

\coordinate (0000110) at (9,2);
\coordinate (1000110) at (9,3);
\coordinate (1100110) at (9,4);
\coordinate (1100111) at (9,5);


\path[edge_chain] (0000000) to[] (1000000);
\path[edge_chain] (1000000) to[] (1100000);
\path[edge_chain] (1100000) to[] (1110000);
\path[edge_chain] (1110000) to[] (1111000);
\path[edge_chain] (1111000) to[] (1111100);
\path[edge_chain] (1111100) to[] (1111110);
\path[edge_chain] (1111110) to[] (1111111);

\path[edge_chain] (0010000) to[] (1010000);
\path[edge_chain] (1010000) to[] (1011000);
\path[edge_chain] (1011000) to[] (1011100);
\path[edge_chain] (1011100) to[] (1011110);
\path[edge_chain] (1011110) to[] (1011111);

\path[edge_chain] (0000010) to[] (1000010);
\path[edge_chain] (1000010) to[] (1100010);
\path[edge_chain] (1100010) to[] (1110010);
\path[edge_chain] (1110010) to[] (1111010);
\path[edge_chain] (1111010) to[] (1111011);

\path[edge_chain] (0001000) to[] (1001000);
\path[edge_chain] (1001000) to[] (1101000);
\path[edge_chain] (1101000) to[] (1101100);
\path[edge_chain] (1101100) to[] (1101110);
\path[edge_chain] (1101110) to[] (1101111);

\path[edge_chain] (0000100) to[] (1000100);
\path[edge_chain] (1000100) to[] (1100100);
\path[edge_chain] (1100100) to[] (1110100);
\path[edge_chain] (1110100) to[] (1110110);
\path[edge_chain] (1110110) to[] (1110111);

\path[edge_chain] (0010100) to[] (1010100);
\path[edge_chain] (1010100) to[] (1010110);
\path[edge_chain] (1010110) to[] (1010111);

\path[edge_chain] (0001010) to[] (1001010);
\path[edge_chain] (1001010) to[] (1101010);
\path[edge_chain] (1101010) to[] (1101011);

\path[edge_chain] (0010010) to[] (1010010);
\path[edge_chain] (1010010) to[] (1011010);
\path[edge_chain] (1011010) to[] (1011011);

\path[edge_chain] (0001100) to[] (1001100);
\path[edge_chain] (1001100) to[] (1001110);
\path[edge_chain] (1001110) to[] (1001111);

\path[edge_chain] (0000110) to[] (1000110);
\path[edge_chain] (1000110) to[] (1100110);
\path[edge_chain] (1100110) to[] (1100111);

\def\edgeoffset{0pt}


\path[edge2, solid, line width = \linew] (0000000) edge[transform canvas={xshift=-\edgeoffset}] (1000000);
\path[edge2, solid, line width = \linew] (1000000) edge[transform canvas={xshift=-\edgeoffset}] (1100000);
\path[edge2, solid, line width = \linew] (1100000) edge[transform canvas={xshift=-\edgeoffset}] (1110000);
\path[edge2, solid, line width = \linew] (1110000) edge[transform canvas={xshift=-\edgeoffset}] (1111000);
\path[edge2, solid, line width = \linew] (1111000) edge[transform canvas={xshift=-\edgeoffset}] (1111100);
\path[edge2, solid, line width = \linew] (1111100) edge[transform canvas={xshift=-\edgeoffset}] (1111110);
\path[edge2, solid, line width = \linew] (1111110) edge[transform canvas={xshift=-\edgeoffset}] (1111111);

\path[edge2, solid, line width = \linew] (1000010) to[] (1100010);
\path[edge2, solid, line width = \linew] (1100010) to[] (1110010);
\path[edge2, solid, line width = \linew] (1110010) to[] (1111010);

\path[edge2, solid, line width = \linew] (1000100) to[] (1100100);
\path[edge2, solid, line width = \linew] (1100100) to[] (1110100);
\path[edge2, solid, line width = \linew] (1110100) to[] (1110110);

\path[edge2, solid, line width = \linew] (1101000) edge[transform canvas={xshift=-\edgeoffset}] (1101100);

\path[edge2, solid, line width = \linew] (1001010) to[] (1101010);

\node[node_square, label={[node_cube_label]below:\small0000000}] (n0000000)  at (0000000) {};
\node[node_cube, label={[node_cube_label]below:\small0000010}] (n0000010)  at (0000010) {};
\node[node_cube, label={[node_cube_label]below:\small0000100}] (n0000100)  at (0000100) {};
\node[node_cube, label={[node_cube_label]below:\small0000110}] (n0000110)  at (0000110) {};
\node[node_cube, label={[node_cube_label]below:\small0001000}] (n0001000)  at (0001000) {};
\node[node_cube, label={[node_cube_label]below:\small0001010}] (n0001010)  at (0001010) {};
\node[node_cube, label={[node_cube_label]below:\small0001100}] (n0001100)  at (0001100) {};
\node[node_cube, label={[node_cube_label]below:\small0010000}] (n0010000)  at (0010000) {};
\node[node_cube, label={[node_cube_label]below:\small0010010}] (n0010010)  at (0010010) {};
\node[node_cube, label={[node_cube_label]below:\small0010100}] (n0010100)  at (0010100) {};
\node[node_square, label={[node_cube_label]below:\small1000000}] (n1000000)  at (1000000) {};
\node[node_square, label={[node_cube_label]below:\small1000010}] (n1000010)  at (1000010) {};
\node[node_square, label={[node_cube_label]below:\small1000100}] (n1000100)  at (1000100) {};
\node[node_cube, label={[node_cube_label]below:\small1000110}] (n1000110)  at (1000110) {};
\node[node_square, label={[node_cube_label]below:\small\cancel{1001000}}] (n1001000)  at (1001000) {};
\node[node_square, label={[node_cube_label]below:\small1001010}] (n1001010)  at (1001010) {};
\node[node_cube, label={[node_cube_label]below:\small1001100}] (n1001100)  at (1001100) {};
\node[node_cube, label={[node_cube_label]below:\small1001110}] (n1001110)  at (1001110) {};
\node[node_cube, label={[node_cube_label]below:\small1001111}] (n1001111)  at (1001111) {};
\node[node_square, label={[node_cube_label]below:\small\cancel{1010000}}] (n1010000)  at (1010000) {};
\node[node_square, label={[node_cube_label]below:\small\cancel{1010010}}] (n1010010)  at (1010010) {};
\node[node_square, label={[node_cube_label]below:\small\cancel{1010100}}] (n1010100)  at (1010100) {};
\node[node_square, label={[node_cube_label]below:\small\cancel{1010110}}] (n1010110)  at (1010110) {};
\node[node_cube, label={[node_cube_label]below:\small1010111}] (n1010111)  at (1010111) {};
\node[node_square, label={[node_cube_label]below:\small\cancel{1011000}}] (n1011000)  at (1011000) {};
\node[node_square, label={[node_cube_label]below:\small\cancel{1011010}}] (n1011010)  at (1011010) {};
\node[node_cube, label={[node_cube_label]below:\small1011011}] (n1011011)  at (1011011) {};
\node[node_square, label={[node_cube_label]below:\small\cancel{1011100}}] (n1011100)  at (1011100) {};
\node[node_square, label={[node_cube_label]below:\small\cancel{1011110}}] (n1011110)  at (1011110) {};
\node[node_cube, label={[node_cube_label]below:\small1011111}] (n1011111)  at (1011111) {};
\node[node_square, label={[node_cube_label]below:\small1100000}] (n1100000)  at (1100000) {};
\node[node_square, label={[node_cube_label]below:\small1100010}] (n1100010)  at (1100010) {};
\node[node_square, label={[node_cube_label]below:\small1100100}] (n1100100)  at (1100100) {};
\node[node_cube, label={[node_cube_label]below:\small1100110}] (n1100110)  at (1100110) {};
\node[node_cube, label={[node_cube_label]below:\small1100111}] (n1100111)  at (1100111) {};
\node[node_square, label={[node_cube_label]below:\small1101000}] (n1101000)  at (1101000) {};
\node[node_square, label={[node_cube_label]below:\small1101010}] (n1101010)  at (1101010) {};
\node[node_cube, label={[node_cube_label]below:\small1101011}] (n1101011)  at (1101011) {};
\node[node_square, label={[node_cube_label]below:\small1101100}] (n1101100)  at (1101100) {};
\node[node_square, label={[node_cube_label]below:\small\cancel{1101110}}] (n1101110)  at (1101110) {};
\node[node_cube, label={[node_cube_label]below:\small1101111}] (n1101111)  at (1101111) {};
\node[node_square, label={[node_cube_label]below:\small1110000}] (n1110000)  at (1110000) {};
\node[node_square, label={[node_cube_label]below:\small1110010}] (n1110010)  at (1110010) {};
\node[node_square, label={[node_cube_label]below:\small1110100}] (n1110100)  at (1110100) {};
\node[node_square, label={[node_cube_label]below:\small1110110}] (n1110110)  at (1110110) {};
\node[node_cube, label={[node_cube_label]below:\small1110111}] (n1110111)  at (1110111) {};
\node[node_square, label={[node_cube_label]below:\small1111000}] (n1111000)  at (1111000) {};
\node[node_square, label={[node_cube_label]below:\small1111010}] (n1111010)  at (1111010) {};
\node[node_cube, label={[node_cube_label]below:\small1111011}] (n1111011)  at (1111011) {};
\node[node_square, label={[node_cube_label]below:\small1111100}] (n1111100)  at (1111100) {};
\node[node_square, label={[node_cube_label]below:\small1111110}] (n1111110)  at (1111110) {};
\node[node_square, label={[node_cube_label]below:\small1111111}] (n1111111)  at (1111111) {};

\path[equiv] (n1010000) to[] (n1000010);
\path[equiv] (n1011000) to[] (n1100010);
\path[equiv] (n1011100) to[] (n1110010);
\path[equiv] (n1011110) to[] (n1111010);

\path[equiv] (n1001000) to[] (n1000100);
\path[equiv] (n1101110) to[] (n1110110);

\path[equiv] (n1010100) to[] (n1001010);
\path[equiv] (n1010110) to[] (n1101010);
\path[equiv] (n1001010) to[] (n1010010);
\path[equiv] (n1101010) to[] (n1011010);

\end{tikzpicture}

%% file: figures/compl.tex
\begin{tikzpicture}[xscale = 1,yscale=.75]

\def\first{.5}
\def\firstdot{1}
\def\pone{1.5}
\def\p{1.8}
\def\middot{2.3}
\def\q{2.8}
\def\qone{3.1}
\def\lastdot{3.6}
\def\last{4.1}

\def\xoffset{.15}%
\def\yoffsetm{.35}%

\node[] (i) at (\first,.75) {$1$};
\node[] (i) at (\p,.75) {$i$};
\node[] (i) at (\q,.75) {$j$};
\node[] (i) at (\last,.75) {$n$};

\begin{scope}[yshift = 0cm]
\node[] at (0,0) {$\strut y = $};
\node[] at (\first,0) {$\strut1$};
\node[] at (\firstdot,0) {$*\!*\!*$};
\node[] at (\pone,0) {$\strut1$};
\node[] at (\p,0) {$\strut1$};
\node[] at (\middot,0) {$*\!*\!*$};
\node[] at (\q,0) {$\strut1$};
\node[] at (\qone,0) {$\strut0$};
\node[] at (\lastdot,0) {$*\!*\!*$};
\node[] at (\last,0) {$\strut0$};
\end{scope}

\begin{scope}[yshift = -0.7cm]
\node[] at (0,0) {$\strut x = $};
\node[] at (\first,0) {$\strut1$};
\node[] at (\firstdot,0) {$*\!*\!*$};
\node[] at (\pone,0) {$\strut1$};
\node[] at (\p,0) {$\strut0$};
\node[] at (\middot,0) {$*\!*\!*$};
\node[] at (\q,0) {$\strut0$};
\node[] at (\qone,0) {$\strut0$};
\node[] at (\lastdot,0) {$*\!*\!*$};
\node[] at (\last,0) {$\strut0$};
\end{scope}

\begin{scope}[xshift = 7cm, yshift= 0cm]
\node[] (i) at (\first,.75) {$1$};
\node[] (i) at (\p,.75) {$i$};
\node[] (i) at (\q,.75) {$j$};
\node[] (i) at (\last,.75) {$n$};
\end{scope}

\begin{scope}[xshift = 7cm]
\begin{scope}[yshift = 0cm]
\path[draw = black!20!white, fill=black!20!white, thick, rounded corners =1, line cap = round] (\pone - \xoffset,-1) to[] (\pone - \xoffset,\yoffsetm) to[] (\p + \xoffset,\yoffsetm) to[] (\p + \xoffset,-1) -- cycle;
\path[draw = black!20!white, fill=black!20!white, thick, rounded corners =1, line cap = round] (\q - \xoffset,-1) to[] (\q - \xoffset,\yoffsetm) to[] (\qone + \xoffset,\yoffsetm) to[] (\qone + \xoffset,-1) -- cycle;
\node[] at (0,0) {$\strut \ol{y} = $};
\node[] at (\first,0) {$\strut0$};
\node[] at (\firstdot,0) {$*\!*\!*$};
\node[] at (\pone,0) {$\strut0$};
\node[] at (\p,0) {$\strut0$};
\node[] at (\middot,0) {$*\!*\!*$};
\node[] at (\q,0) {$\strut0$};
\node[] at (\qone,0) {$\strut1$};
\node[] at (\lastdot,0) {$*\!*\!*$};
\node[] at (\last,0) {$\strut1$};
\end{scope}

\begin{scope}[yshift= -0.7cm]
\node[] at (0,0) {$\strut \ol{x} = $};
\node[] at (\first,0) {$\strut0$};
\node[] at (\firstdot,0) {$*\!*\!*$};
\node[] at (\pone,0) {$\strut0$};
\node[] at (\p,0) {$\strut1$};
\node[] at (\middot,0) {$*\!*\!*$};
\node[] at (\q,0) {$\strut1$};
\node[] at (\qone,0) {$\strut1$};
\node[] at (\lastdot,0) {$*\!*\!*$};
\node[] at (\last,0) {$\strut1$};
\end{scope}
\end{scope}

\end{tikzpicture}

%% file: figures/unimodal.tex
\begin{tikzpicture}[xscale = 1,yscale=1,
bracketintra/.style={draw, thick, solid},
bracketinter/.style={draw, thick, dash pattern={{on 1pt off 1pt}}}]

\def\first{.5}
\def\offset{.3}

\def\ycenter{4pt}
\def\yoffset{6pt}

\newcommand{\DrawBrac}[4][]{%
    \draw [#1, rounded corners=2pt]  ([yshift=-\ycenter]#3.center)%
           -- ([yshift=-\yoffset*#2]#3.center)%
           -- ([yshift=-\yoffset*#2]#4.center)%
           -- ([yshift=-\ycenter]#4.center)%
}

\node[] at (0,0) {$\strut y = $};
\foreach \v in {0,1,2} {
\foreach \b [count=\i from 0] in {1,1,1,1,0,0,0,1,0,1,1,0} {
  \node[] (y\v\i) at (\first + \v * 12 * \offset + \i * \offset, 0) {$\strut\b$};
}
}

\node[] at (0,-1) {$\strut x = $};
\foreach \v in {0,1,2} {
\foreach \b [count=\i from 0] in {1,1,0,0,0,0,0,1,0,1,1,0} {
  \node[] (x\v\i) at (\first + \v * 12 * \offset + \i * \offset, -1) {$\strut\b$};
}
}

\node[] at (0,-2) {$\strut z = $};
\foreach \v in {0} {
\foreach \b [count=\i from 0] in {1,0,0,0,0,0,0,1,0,1,1,0} {
  \node[] (z\v\i) at (\first + \v * 12 * \offset + \i * \offset, -2) {$\strut\b$};
}
}
\foreach \v in {1,2} {
\foreach \b [count=\i from 0] in {1,1,0,0,0,0,0,1,0,1,1,0} {
  \node[] (z\v\i) at (\first + \v * 12 * \offset + \i * \offset, -2) {$\strut\b$};
}
}

\draw[thick] (\first + 12 * \offset -\offset/2, -2.5) to[] (\first + 12 * \offset -\offset/2, .5);
\draw[thick] (\first + 2 * 12 * \offset -\offset/2, -2.5) to[] (\first + 2 * 12 * \offset -\offset/2, .5);

\DrawBrac[bracketintra]{1}{y06}{y07};
\DrawBrac[bracketintra]{1}{y08}{y09};
\DrawBrac[bracketintra]{2}{y05}{y010};
\DrawBrac[bracketinter]{1}{y011}{y10};
\DrawBrac[bracketinter]{3}{y04}{y11};

\DrawBrac[bracketintra]{1}{y16}{y17};
\DrawBrac[bracketintra]{1}{y18}{y19};
\DrawBrac[bracketintra]{2}{y15}{y110};
\DrawBrac[bracketinter]{1}{y111}{y20};
\DrawBrac[bracketinter]{3}{y14}{y21};

\DrawBrac[bracketintra]{1}{y26}{y27};
\DrawBrac[bracketintra]{1}{y28}{y29};
\DrawBrac[bracketintra]{2}{y25}{y210};

\DrawBrac[bracketintra]{1}{x06}{x07};
\DrawBrac[bracketintra]{1}{x08}{x09};
\DrawBrac[bracketintra]{2}{x05}{x010};
\DrawBrac[bracketinter]{1}{x011}{x10};
\DrawBrac[bracketinter]{3}{x04}{x11};

\DrawBrac[bracketintra]{1}{x16}{x17};
\DrawBrac[bracketintra]{1}{x18}{x19};
\DrawBrac[bracketintra]{2}{x15}{x110};
\DrawBrac[bracketinter]{1}{x111}{x20};
\DrawBrac[bracketinter]{3}{x14}{x21};

\DrawBrac[bracketintra]{1}{x26}{x27};
\DrawBrac[bracketintra]{1}{x28}{x29};
\DrawBrac[bracketintra]{2}{x25}{x210};

\DrawBrac[bracketintra]{1}{z06}{z07};
\DrawBrac[bracketintra]{1}{z08}{z09};
\DrawBrac[bracketintra]{2}{z05}{z010};
\DrawBrac[bracketinter]{1}{z011}{z10};
\DrawBrac[bracketinter]{3}{z04}{z11};

\DrawBrac[bracketintra]{1}{z16}{z17};
\DrawBrac[bracketintra]{1}{z18}{z19};
\DrawBrac[bracketintra]{2}{z15}{z110};
\DrawBrac[bracketinter]{1}{z111}{z20};
\DrawBrac[bracketinter]{3}{z14}{z21};

\DrawBrac[bracketintra]{1}{z26}{z27};
\DrawBrac[bracketintra]{1}{z28}{z29};
\DrawBrac[bracketintra]{2}{z25}{z210};

\draw [decorate,decoration={brace,raise=30pt}](\first -\offset/2,0) -- node[above=32pt] {$n = 36$} (\first + 3*12*\offset -\offset/2,0);
\draw [decorate,decoration={brace,raise=10pt}](\first -\offset/2,0) -- node[above=12pt] {$d=12$} (\first + 12*\offset -\offset/2,0);
\end{tikzpicture}

%% file: figures/N6m.tex
\begin{tikzpicture}[xscale = 1.75,
edgegraph/.style={draw=black!20!white, thick}]

\coordinate (000000) at (0.000000,0);
\coordinate (000001) at (0.000000,1);
\coordinate (000011) at (0.000000,2);
\coordinate (000101) at (1.000000,2);
\coordinate (000111) at (0.000000,3);
\coordinate (001001) at (2.000000,2);
\coordinate (010011) at (1.000000,3);
\coordinate (001011) at (2.000000,3);
\coordinate (001111) at (0.000000,4);
\coordinate (010101) at (3.000000,3);
\coordinate (010111) at (1.000000,4);
\coordinate (011011) at (2.000000,4);
\coordinate (011111) at (0.000000,5);
\coordinate (111111) at (0.000000,6);

\path[edgegraph] (111111) to[out = -90.000000, in = 180 - -90.000000, relative] (011111);
\path[edgegraph] (111111) to[out = -40.000000, in = 180 - -40.000000, relative] (011111);
\path[edgegraph] (111111) to[out = -15.000000, in = 180 - -15.000000, relative] (011111);
\path[edgegraph] (111111) to[out = 15.000000, in = 180 - 15.000000, relative] (011111);
\path[edgegraph] (111111) to[out = 40.000000, in = 180 - 40.000000, relative] (011111);
\path[edgegraph] (111111) to[out = 90.000000, in = 180 - 90.000000, relative] (011111);
\path[edgegraph] (011111) to[out = -20.000000, in = 180 - -20.000000, relative] (001111);
\path[edgegraph] (011111) to[out = 20.000000, in = 180 - 20.000000, relative] (001111);
\path[edgegraph] (011111) to[out = -20.000000, in = 180 - -20.000000, relative] (010111);
\path[edgegraph] (011111) to[out = 20.000000, in = 180 - 20.000000, relative] (010111);
\path[edgegraph] (011011) to[out = -20.000000, in = 180 - -20.000000, relative] (010011);
\path[edgegraph] (011011) to[out = 20.000000, in = 180 - 20.000000, relative] (010011);
\path[edgegraph] (011011) to[out = -20.000000, in = 180 - -20.000000, relative] (001011);
\path[edgegraph] (011011) to[out = 20.000000, in = 180 - 20.000000, relative] (001011);
\path[edgegraph] (010111) to[out = 0.000000, in = 180 - 0.000000, relative] (001011);
\path[edgegraph] (010111) to[out = 0.000000, in = 180 - 0.000000, relative] (000111);
\path[edgegraph] (010111) to[out = 0.000000, in = 180 - 0.000000, relative] (010011);
\path[edgegraph] (001111) to[out = 0.000000, in = 180 - 0.000000, relative] (010011);
\path[edgegraph] (001111) to[out = -20.000000, in = 180 - -20.000000, relative] (000111);
\path[edgegraph] (001111) to[out = 20.000000, in = 180 - 20.000000, relative] (000111);
\path[edgegraph] (001111) to[out = 0.000000, in = 180 - 0.000000, relative] (001011);
\path[edgegraph] (000001) to[out = -20.000000, in = 180 - -20.000000, relative] (000011);
\path[edgegraph] (000001) to[out = 20.000000, in = 180 - 20.000000, relative] (000011);
\path[edgegraph] (000001) to[out = -20.000000, in = 180 - -20.000000, relative] (000101);
\path[edgegraph] (000001) to[out = 20.000000, in = 180 - 20.000000, relative] (000101);
\path[edgegraph] (000000) to[out = -90.000000, in = 180 - -90.000000, relative] (000001);
\path[edgegraph] (000000) to[out = -40.000000, in = 180 - -40.000000, relative] (000001);
\path[edgegraph] (000000) to[out = -15.000000, in = 180 - -15.000000, relative] (000001);
\path[edgegraph] (000000) to[out = 15.000000, in = 180 - 15.000000, relative] (000001);
\path[edgegraph] (000000) to[out = 40.000000, in = 180 - 40.000000, relative] (000001);
\path[edgegraph] (000000) to[out = 90.000000, in = 180 - 90.000000, relative] (000001);
\path[edgegraph] (000011) to[out = 0.000000, in = 180 - 0.000000, relative] (001011);
\path[edgegraph] (000011) to[out = -20.000000, in = 180 - -20.000000, relative] (000111);
\path[edgegraph] (000011) to[out = 20.000000, in = 180 - 20.000000, relative] (000111);
\path[edgegraph] (000011) to[out = 0.000000, in = 180 - 0.000000, relative] (010011);
\path[edgegraph] (000101) to[out = 0.000000, in = 180 - 0.000000, relative] (010011);
\path[edgegraph] (000101) to[out = 0.000000, in = 180 - 0.000000, relative] (001011);
\path[edgegraph] (000101) to[out = 0.000000, in = 180 - 0.000000, relative] (000111);
\path[edgegraph] (001001) to[out = -20.000000, in = 180 - -20.000000, relative] (001011);
\path[edgegraph] (001001) to[out = 20.000000, in = 180 - 20.000000, relative] (001011);
\path[edgegraph] (001001) to[out = -20.000000, in = 180 - -20.000000, relative] (010011);
\path[edgegraph] (001001) to[out = 20.000000, in = 180 - 20.000000, relative] (010011);

\node[node_cube, deficient, label={[black, node_cube_label]left:\small$\neck{\{1,\!2,\!3,\!4,\!5,\!6\}}$}] at (111111) {};
\node[node_cube, full, label={[black, node_cube_label]left:\small$\neck{\{1,\!2,\!3,\!4,\!5\}}$}] at (011111) {};
\node[node_cube, deficient, label={[black, node_cube_label]above:\small$\neck{\{1,\!2,\!4,\!5\}}$}] at (011011) {};
\node[node_cube, full, label={[black, node_cube_label]above:\small$\neck{\{1,\!2,\!3,\!5\}}$}] at (010111) {};
\node[node_cube, deficient, label={[black, node_cube_label]below:\small$\neck{\{1,\!3,\!5\}}$}] at (010101) {};
\node[node_cube, full, label={[black, node_cube_label]left:\small$\neck{\{1,\!2,\!3,\!4\}}$}] at (001111) {};
\node[node_cube, full, label={[black, node_cube_label]below:\small$\neck{\{1,\!2,\!4\}}$}] at (001011) {};
\node[node_cube, full, label={[black, node_cube_label]below:\small$\neck{\{1,\!2,\!5\}}$}] at (010011) {};
\node[node_cube, full, label={[black, node_cube_label]left:\small$\neck{\{1\}}$}] at (000001) {};
\node[node_cube, deficient, label={[black, node_cube_label]left:\small$\neck{\emptyset}$}] at (000000) {};
\node[node_cube, full, label={[black, node_cube_label]left:\small$\neck{\{1,\!2,\!3\}}$}] at (000111) {};
\node[node_cube, full, label={[black, node_cube_label]left:\small$\neck{\{1,\!2\}}$}] at (000011) {};
\node[node_cube, full, label={[black, node_cube_label]below:\small$\neck{\{1,\!3\}}$}] at (000101) {};
\node[node_cube, deficient, label={[black, node_cube_label]below:\small$\neck{\{1,\!4\}}$}] at (001001) {};
\end{tikzpicture}

%% file: figures/VW7.tex
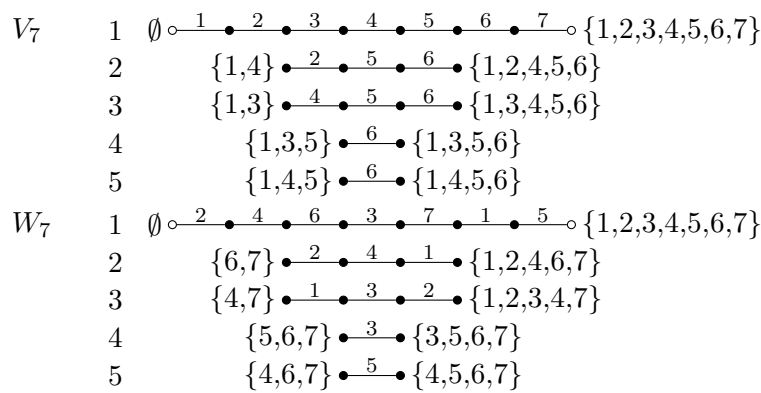
\begin{figure}[h!]
\centering
\begin{tikzpicture}[scale=1,
nodelabel/.style={inner sep=0pt},
node1/.style={draw=black,fill=black,
circle,inner sep=0pt,text width=0pt,text height=0pt,
text depth=0pt,minimum size=3pt},
full/.style={fill=black},
deficient/.style={fill=white},
edgelabel/.style={above=1pt, inner sep=0pt},
path1/.style={draw}]

\def\xs{.75}
\def\ys{.5}

\node[anchor=west] (label0)  at (\xs*-3.000000,\ys*0.000000) {$V_7$};
\node[] (path1)  at (\xs*-1.000000,\ys*0.000000) {1};
\node[nodelabel,anchor=east] (n10)  at (\xs*0.000000,\ys*0.000000) {$\emptyset\ $};
\node[node1,deficient] (n10)  at (\xs*0.000000,\ys*0.000000) {};
\node[node1,full] (n11)  at (\xs*1.000000,\ys*0.000000) {};
\node[node1,full] (n12)  at (\xs*2.000000,\ys*0.000000) {};
\node[node1,full] (n13)  at (\xs*3.000000,\ys*0.000000) {};
\node[node1,full] (n14)  at (\xs*4.000000,\ys*0.000000) {};
\node[node1,full] (n15)  at (\xs*5.000000,\ys*0.000000) {};
\node[node1,full] (n16)  at (\xs*6.000000,\ys*0.000000) {};
\node[nodelabel,anchor=west] (n17)  at (\xs*7.000000,\ys*0.000000) {$\ \{1,\!2,\!3,\!4,\!5,\!6,\!7\}$};
\node[node1,deficient] (n17)  at (\xs*7.000000,\ys*0.000000) {};

\path[path1] (n10)  to[] node[edgelabel] {\scriptsize$1$} (n11);
\path[path1] (n11)  to[] node[edgelabel] {\scriptsize$2$} (n12);
\path[path1] (n12)  to[] node[edgelabel] {\scriptsize$3$} (n13);
\path[path1] (n13)  to[] node[edgelabel] {\scriptsize$4$} (n14);
\path[path1] (n14)  to[] node[edgelabel] {\scriptsize$5$} (n15);
\path[path1] (n15)  to[] node[edgelabel] {\scriptsize$6$} (n16);
\path[path1] (n16)  to[] node[edgelabel] {\scriptsize$7$} (n17);

\node[] (path2)  at (\xs*-1.000000,\ys*-1.000000) {2};
\node[nodelabel,anchor=east] (n20)  at (\xs*2.000000,\ys*-1.000000) {$\{1,\!4\}\ $};
\node[node1,full] (n20)  at (\xs*2.000000,\ys*-1.000000) {};
\node[node1,full] (n21)  at (\xs*3.000000,\ys*-1.000000) {};
\node[node1,full] (n22)  at (\xs*4.000000,\ys*-1.000000) {};
\node[nodelabel,anchor=west] (n23)  at (\xs*5.000000,\ys*-1.000000) {$\ \{1,\!2,\!4,\!5,\!6\}$};
\node[node1,full] (n23)  at (\xs*5.000000,\ys*-1.000000) {};

\path[path1] (n20)  to[] node[edgelabel] {\scriptsize$2$} (n21);
\path[path1] (n21)  to[] node[edgelabel] {\scriptsize$5$} (n22);
\path[path1] (n22)  to[] node[edgelabel] {\scriptsize$6$} (n23);

\node[] (path3)  at (\xs*-1.000000,\ys*-2.000000) {3};
\node[nodelabel,anchor=east] (n30)  at (\xs*2.000000,\ys*-2.000000) {$\{1,\!3\}\ $};
\node[node1,full] (n30)  at (\xs*2.000000,\ys*-2.000000) {};
\node[node1,full] (n31)  at (\xs*3.000000,\ys*-2.000000) {};
\node[node1,full] (n32)  at (\xs*4.000000,\ys*-2.000000) {};
\node[nodelabel,anchor=west] (n33)  at (\xs*5.000000,\ys*-2.000000) {$\ \{1,\!3,\!4,\!5,\!6\}$};
\node[node1,full] (n33)  at (\xs*5.000000,\ys*-2.000000) {};

\path[path1] (n30)  to[] node[edgelabel] {\scriptsize$4$} (n31);
\path[path1] (n31)  to[] node[edgelabel] {\scriptsize$5$} (n32);
\path[path1] (n32)  to[] node[edgelabel] {\scriptsize$6$} (n33);

\node[] (path4)  at (\xs*-1.000000,\ys*-3.000000) {4};
\node[nodelabel,anchor=east] (n40)  at (\xs*3.000000,\ys*-3.000000) {$\{1,\!3,\!5\}\ $};
\node[node1,full] (n40)  at (\xs*3.000000,\ys*-3.000000) {};
\node[nodelabel,anchor=west] (n41)  at (\xs*4.000000,\ys*-3.000000) {$\ \{1,\!3,\!5,\!6\}$};
\node[node1,full] (n41)  at (\xs*4.000000,\ys*-3.000000) {};

\path[path1] (n40)  to[] node[edgelabel] {\scriptsize$6$} (n41);

\node[] (path5)  at (\xs*-1.000000,\ys*-4.000000) {5};
\node[nodelabel,anchor=east] (n50)  at (\xs*3.000000,\ys*-4.000000) {$\{1,\!4,\!5\}\ $};
\node[node1,full] (n50)  at (\xs*3.000000,\ys*-4.000000) {};
\node[nodelabel,anchor=west] (n51)  at (\xs*4.000000,\ys*-4.000000) {$\ \{1,\!4,\!5,\!6\}$};
\node[node1,full] (n51)  at (\xs*4.000000,\ys*-4.000000) {};

\path[path1] (n50)  to[] node[edgelabel] {\scriptsize$6$} (n51);

\end{tikzpicture}

\begin{tikzpicture}[scale=1,
nodelabel/.style={inner sep=0pt},
node1/.style={draw=black,fill=black,
circle,inner sep=0pt,text width=0pt,text height=0pt,
text depth=0pt,minimum size=3pt},
full/.style={fill=black},
deficient/.style={fill=white},
edgelabel/.style={above=1pt, inner sep=0pt},
path1/.style={draw}]

\def\xs{.75}
\def\ys{.5}

\node[anchor=west] (label2)  at (\xs*-3.000000,\ys*0.000000) {$W_7$};
\node[] (path1)  at (\xs*-1.000000,\ys*0.000000) {1};
\node[nodelabel,anchor=east] (n10)  at (\xs*0.000000,\ys*0.000000) {$\emptyset\ $};
\node[node1,deficient] (n10)  at (\xs*0.000000,\ys*0.000000) {};
\node[node1,full] (n11)  at (\xs*1.000000,\ys*0.000000) {};
\node[node1,full] (n12)  at (\xs*2.000000,\ys*0.000000) {};
\node[node1,full] (n13)  at (\xs*3.000000,\ys*0.000000) {};
\node[node1,full] (n14)  at (\xs*4.000000,\ys*0.000000) {};
\node[node1,full] (n15)  at (\xs*5.000000,\ys*0.000000) {};
\node[node1,full] (n16)  at (\xs*6.000000,\ys*0.000000) {};
\node[nodelabel,anchor=west] (n17)  at (\xs*7.000000,\ys*0.000000) {$\ \{1,\!2,\!3,\!4,\!5,\!6,\!7\}$};
\node[node1,deficient] (n17)  at (\xs*7.000000,\ys*0.000000) {};

\path[path1] (n10)  to[] node[edgelabel] {\scriptsize$2$} (n11);
\path[path1] (n11)  to[] node[edgelabel] {\scriptsize$4$} (n12);
\path[path1] (n12)  to[] node[edgelabel] {\scriptsize$6$} (n13);
\path[path1] (n13)  to[] node[edgelabel] {\scriptsize$3$} (n14);
\path[path1] (n14)  to[] node[edgelabel] {\scriptsize$7$} (n15);
\path[path1] (n15)  to[] node[edgelabel] {\scriptsize$1$} (n16);
\path[path1] (n16)  to[] node[edgelabel] {\scriptsize$5$} (n17);

\node[] (path2)  at (\xs*-1.000000,\ys*-1.000000) {2};
\node[nodelabel,anchor=east] (n20)  at (\xs*2.000000,\ys*-1.000000) {$\{6,\!7\}\ $};
\node[node1,full] (n20)  at (\xs*2.000000,\ys*-1.000000) {};
\node[node1,full] (n21)  at (\xs*3.000000,\ys*-1.000000) {};
\node[node1,full] (n22)  at (\xs*4.000000,\ys*-1.000000) {};
\node[nodelabel,anchor=west] (n23)  at (\xs*5.000000,\ys*-1.000000) {$\ \{1,\!2,\!4,\!6,\!7\}$};
\node[node1,full] (n23)  at (\xs*5.000000,\ys*-1.000000) {};

\path[path1] (n20)  to[] node[edgelabel] {\scriptsize$2$} (n21);
\path[path1] (n21)  to[] node[edgelabel] {\scriptsize$4$} (n22);
\path[path1] (n22)  to[] node[edgelabel] {\scriptsize$1$} (n23);

\node[] (path3)  at (\xs*-1.000000,\ys*-2.000000) {3};
\node[nodelabel,anchor=east] (n30)  at (\xs*2.000000,\ys*-2.000000) {$\{4,\!7\}\ $};
\node[node1,full] (n30)  at (\xs*2.000000,\ys*-2.000000) {};
\node[node1,full] (n31)  at (\xs*3.000000,\ys*-2.000000) {};
\node[node1,full] (n32)  at (\xs*4.000000,\ys*-2.000000) {};
\node[nodelabel,anchor=west] (n33)  at (\xs*5.000000,\ys*-2.000000) {$\ \{1,\!2,\!3,\!4,\!7\}$};
\node[node1,full] (n33)  at (\xs*5.000000,\ys*-2.000000) {};

\path[path1] (n30)  to[] node[edgelabel] {\scriptsize$1$} (n31);
\path[path1] (n31)  to[] node[edgelabel] {\scriptsize$3$} (n32);
\path[path1] (n32)  to[] node[edgelabel] {\scriptsize$2$} (n33);

\node[] (path4)  at (\xs*-1.000000,\ys*-3.000000) {4};
\node[nodelabel,anchor=east] (n40)  at (\xs*3.000000,\ys*-3.000000) {$\{5,\!6,\!7\}\ $};
\node[node1,full] (n40)  at (\xs*3.000000,\ys*-3.000000) {};
\node[nodelabel,anchor=west] (n41)  at (\xs*4.000000,\ys*-3.000000) {$\ \{3,\!5,\!6,\!7\}$};
\node[node1,full] (n41)  at (\xs*4.000000,\ys*-3.000000) {};

\path[path1] (n40)  to[] node[edgelabel] {\scriptsize$3$} (n41);

\node[] (path5)  at (\xs*-1.000000,\ys*-4.000000) {5};
\node[nodelabel,anchor=east] (n50)  at (\xs*3.000000,\ys*-4.000000) {$\{4,\!6,\!7\}\ $};
\node[node1,full] (n50)  at (\xs*3.000000,\ys*-4.000000) {};
\node[nodelabel,anchor=west] (n51)  at (\xs*4.000000,\ys*-4.000000) {$\ \{4,\!5,\!6,\!7\}$};
\node[node1,full] (n51)  at (\xs*4.000000,\ys*-4.000000) {};

\path[path1] (n50)  to[] node[edgelabel] {\scriptsize$5$} (n51);

\end{tikzpicture}
\caption{Two SCDs~$V_7$ and~$W_7$ in the necklace poset~$N_7$ that together with their complements~$\ol{V_7}$ and~$\ol{W_7}$ can be unrolled to four good almost orthgonal SCDs in~$Q_7$.}
\label{fig:Q7_4_ortho}
\end{figure}